\documentclass[11pt,reqno]{amsart}

\usepackage[T1]{fontenc}
\usepackage{mathpazo}             
\usepackage{microtype}            
\usepackage[margin=1.10in]{geometry}

\usepackage{amsmath,amssymb,amsfonts}
\usepackage{mathtools}
\usepackage{mathrsfs}

\usepackage{graphicx}
\usepackage{float} 
\usepackage{booktabs}
\usepackage{tikz}
\usetikzlibrary{
  arrows.meta,
  calc,
  positioning,
  patterns,
  patterns.meta,
  decorations.pathreplacing,
  intersections,
  fit
}

\usepackage{xcolor}
\usepackage[
  colorlinks=true,
  linkcolor=black,
  citecolor=blue!65!black,
  urlcolor=blue!65!black
]{hyperref}

\newtheorem{theorem}{Theorem}[section]
\newtheorem{proposition}[theorem]{Proposition}
\newtheorem{lemma}[theorem]{Lemma}
\newtheorem{corollary}[theorem]{Corollary}

\theoremstyle{definition}
\newtheorem{definition}[theorem]{Definition}

\theoremstyle{remark}
\newtheorem{remark}[theorem]{Remark}

\numberwithin{equation}{section}

\def \Div {\mathrm{div}}

\def \R {\mathbb{R}}

\def \dist {\mathrm{dist}}
\def \loc {\mathrm{loc}}

\def \interior{\mathrm{int}}

\newcommand{\avgint}{\mathchoice{\vcenter{\hbox{$\displaystyle\int$}}\kern-1.15em-}{\vcenter{\hbox{$\textstyle\int$}}\kern-0.95em-}{\int\kern-0.75em-}{\int\kern-0.70em-}}

\title[Tumor growth models with nutrient consumption]{Tumor growth models with nutrient consumption}

\author[D. J. Ara\'ujo]{Dami\~ao J. Ara\'ujo}
\address{Department of Mathematics, Universidade Federal da Para\'iba, 58059-900, Jo\~ao Pessoa-PB, Brazil}{}
\email{araujo@mat.ufpb.br}

\author[R. Teymurazyan]{Rafayel Teymurazyan}
\address{Applied Mathematics and Computational Sciences (AMCS), Compu\-ter, Electrical and Mathematical Sciences and Engineering Division (CEMSE), King Abdullah University of Science and Technology (KAUST), Thuwal, 23955-6900, Kingdom of Saudi Arabia}{}
\email{rafayel.teymurazyan@kaust.edu.sa} 

\author[J.M.~Urbano]{Jos\'{e} Miguel Urbano}
\address{Applied Mathematics and Computational Sciences (AMCS), Compu\-ter, Electrical and Mathematical Sciences and Engineering Division (CEMSE), King Abdullah University of Science and Technology (KAUST), Thuwal, 23955-6900, Kingdom of Saudi Arabia; and CMUC, Department of Mathematics, University of Coimbra, 3000-143 Coimbra, Portugal}{} 
\email{miguel.urbano@kaust.edu.sa} 



\begin{document}

\begin{abstract}
We study the saturation free boundary of the cell density and its interaction with the nutrient depletion interface in a coupled parabolic model for tumor growth with sublinear reactions. We establish a sharp quadratic--superquadratic growth dichotomy along the saturation interface and quantify the degeneration of quadratic profiles as they approach the joint saturation--depletion boundary. We further show that intrinsic regularity of the saturation interface can occur only at joint free boundary points, where half-space blow-ups of either component determine a unique joint blow-up profile. Finally, we prove a local backward coincidence between the joint and nutrient free boundaries near regular points, which yields spatial and temporal regularity of the joint interface.
\end{abstract}

\maketitle

\tableofcontents

\section{Introduction}

In this paper, we study a coupled system of equations that arises as a prototype for tumor growth models with nutrient consumption. Our model is closely related to the class of tumor growth models studied in \cite{DP21, JKT23, MRCS14, PQV14, PTV14}, where the tumor cell density evolves through a pressure mechanism, while the nutrient concentration diffuses and is consumed by the tumor (see also \cite{DPSV26} and the references therein). We focus on the sublinear regime
\begin{equation}\label{cafenenon}
\left\{
\begin{array}{rcr}
\rho_t-\Div(\rho Dp) &=& \rho v^\gamma \\
v_t-\Delta v &=& -\rho v^\theta,
\end{array}
\right.
\end{equation}
where $\rho\geq0$ is the cell density and $v\geq0$ the nutrient concentration, while $p$ denotes the pressure governing the tumor dynamics. In particular, we are interested in the geometry of the tumor interface for sublinear reaction exponents $\gamma, \theta \in (0,1)$. Sublinear absorptions are relevant because they enable the formation of full-cell saturation and nutrient-depleted regions and, consequently, nontrivial dead-cores. This leads to an interaction between the cell density and the nutrient concentration, whose diffusion is determined by cell saturation and nutrient depletion.

A further connection with the classical tumor growth framework arises from the Hele--Shaw limit of porous-medium models, $p=\rho^m$, as $m\to\infty$. The system was studied for finite $m>0$ in \cite{KMMUS97, K97, M04}. In this setting, the density satisfies the saturation constraint $\rho\le1$, and the free boundary $\partial^-\{\rho<1\}$ separates the saturated and unsaturated phases. Near this set, $\rho$ remains bounded away from zero, so the diffusion coefficient $\rho^m$ is strictly positive and the first equation in \eqref{cafenenon} is locally uniformly parabolic. In this context, freezing the coefficients  and introducing the density defect,
$$
u:=1-\rho,
$$
leads to the prototype system
\begin{equation}\label{the problem}
\left\{
\begin{array}{rcr}
\Delta u-u_t &=& v^\gamma\\
\Delta v-v_t &=& v^\theta.
\end{array}
\right.
\end{equation}
Related systems have been extensively studied in the elliptic setting (see, for example, \cite{ALT16, AT24, AT26, DT20, PTU26, T16}). For the parabolic setting, we refer to \cite{AFSW22}, where the regularity of the free boundary for a semilinear problem is investigated.

The main objective of this paper is to investigate the geometric and analytical properties of the density defect $u$ and its \textit{saturated interface}
$$
\Gamma_u:=\partial^-\{u>0\}.
$$
To this end, it is crucial to understand how the behavior of $u$ near $\Gamma_u$ is influenced by the nutrient close to the \textit{depletion interface} $\Gamma_v:=\partial^-\{v>0\}$.
Indeed, the forcing term $v^\gamma$ in the first equation of \eqref{the problem} links the behavior of the density defect to the vanishing rate of $v$ near $\Gamma_v$. Consequently, the local behavior of the nutrient near its free boundary plays a fundamental role in the following approach. This motivates the study of the growth of the pair $(u,v)$ associated with the scaling invariance of \eqref{the problem}, and of the \textit{joint free boundary}
$$
\Gamma:=\partial^-\{|(u,v)|>0\},
$$
where
$$
|(u,v)|:=u^{1/\alpha}+v^{1/\beta},
\quad \mbox{with} \quad
\alpha=2+\gamma\beta,
\quad 
\beta=\frac{2}{1-\theta}.
$$
We first show that $\Gamma_u$ admits a sharp dichotomy into quadratic and superquadratic growth regimes. More precisely, the saturation interface can be decomposed as
$$
\Gamma_u= \mathcal Q \cup \Gamma,
$$
such that, at points of $\mathcal Q:=\Gamma_u\setminus\Gamma$, the nutrient concentration remains strictly positive and the density defect admits a nontrivial quadratic blow-up. In contrast, at joint free boundary points $\Gamma$, the higher intrinsic growth estimate
$$
\sup_{Q_r(z_0)}u\le Cr^\alpha
$$
holds. We further characterize the joint interface through the vanishing of the spatial Hessian and obtain a quantitative decay estimate for the Hessian at points of the nonjoint saturation set $\mathcal Q$
approaching $\Gamma$. These results provide insight into how the behavior of the density defect degenerates near the joint interface. 

Next, we obtain nondegeneracy, space-time porosity, and positive density estimates, showing, in particular, that $\Gamma$ has zero space-time Lebesgue measure. Moreover, denoting by $\mathcal R_u$ and $\mathcal R_v$ the regular sets associated with $\Gamma_u$ and $\Gamma_v$, respectively, we prove that
$$
\mathcal R_u=\Gamma\cap\mathcal R_v=:\mathcal R.
$$
The latter allows us to extend the notion of regular points introduced for the nutrient interface in \cite{AFSW22} to the joint free boundary. Furthermore, we prove that a half-space blow-up of either component uniquely determines the entire joint blow-up,
whose intrinsic size exhibits a linear half-space profile
$$
|(u_{r,z_0},v_{r,z_0})|
\longrightarrow C_{\theta,\gamma}(x\cdot\nu)_+,
$$
for some $C_{\theta,\gamma}>0$. 

Finally, we establish regularity of the joint interface near points of $\mathcal R$. By combining the nutrient free boundary regularity theory of \cite{AFSW22} with
a new backward coincidence argument,
we prove that the joint and nutrient
free boundaries coincide locally in a backward neighborhood. Consequently, the joint interface inherits spatial $C^{1,\sigma}$ and temporal half-H\"older regularity, for some $\sigma\in(0,1)$.

\begin{figure}[ht]
\centering

\begin{tikzpicture}[
    >=Latex,
    font=\small,
    box/.style={
        draw,
        rounded corners=3pt,
        align=center,
        inner sep=8pt,
        text width=5.1cm
    },
    smallbox/.style={
        draw,
        rounded corners=3pt,
        align=center,
        inner sep=7pt,
        text width=4.3cm
    },
    regular/.style={
        draw,
        very thick,
        rounded corners=4pt,
        align=center,
        inner sep=8pt,
        text width=5.4cm
    },
    singular/.style={
        draw,
        dashed,
        rounded corners=3pt,
        align=center,
        inner sep=7pt,
        text width=3.8cm
    },
    arr/.style={
        ->,
        thick
    }
]


\node[box, text width=4.3cm] (Gu) at (0,0) {
    $\Gamma_u$ \textit{cell saturation} \\[4pt]
    \textit{interface}
};


\node[box] (Q) at (-4.0,-3.1) {
    $\mathcal Q:=\Gamma_u\setminus\Gamma$\\[4pt]
    \textit{quadratic regime}
};

\node[box] (G) at (4.0,-3.1) {
    $\Gamma\,$ \textit{joint free boundary}\\[4pt]
    \textit{superquadratic regime} 
};

\coordinate (fork1) at (0,-1.55);

\draw[thick] (Gu.south) -- (fork1);
\draw[thick] (fork1) -| (Q.north);
\draw[thick] (fork1) -| (G.north);

\node[
    font=\footnotesize\itshape,
    fill=white,
    inner sep=2pt
] at (-4.0,-1.55)
    {positive nutrient $v$};

\node[
    font=\footnotesize\itshape,
    fill=white,
    inner sep=2pt
] at (4.0,-1.55)
    {nutrient depletion $\Gamma_v$} ;


\node[
    font=\footnotesize\itshape
] at (2.1,-4.8)
    {\color{gray}free boundary regularity};

\node[box] (R) at (4.0,-6.5) {
    $\mathcal R
      :=\mathcal R_u
      =\Gamma\cap\mathcal R_v
    $ \\[5pt]
    \textit{regular part of $\Gamma$}
};

\draw[thick] (G.south) -- (R.north);


\node[
    font=\footnotesize\itshape,
    below=3pt of R
] {\color{gray}$C^{1,\sigma}$ \textit{in space and} $C^{0,1/2}$ \textit{in time}};

\end{tikzpicture}

\caption{
Schematic organization of the free-boundary analysis.
}
\label{fig:growth-dichotomy-roadmap}
\end{figure}
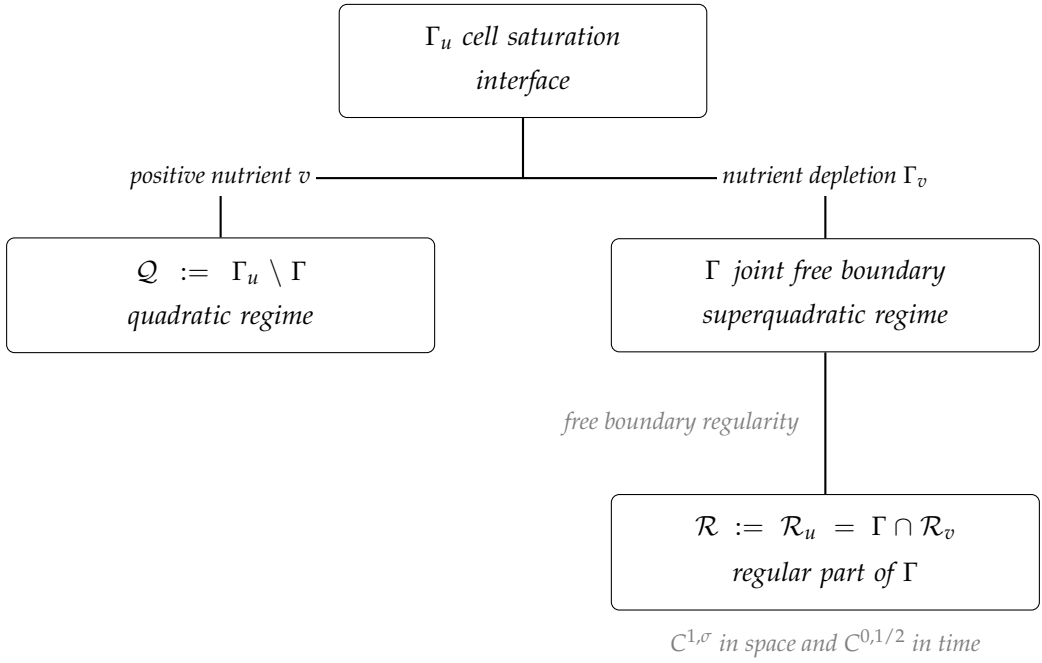

The paper is organized as follows. Section \ref{s2} introduces the parabolic setting, weak solutions, and intrinsic scaling. Section \ref{s4} establishes the qualitative coupling properties between the saturation and depletion interfaces. Section \ref{s3} develops optimal growth estimates and the sharp growth dichotomy
along $\Gamma_u$. Section \ref{s5} proves nondegeneracy,
positive density, and porosity of the joint interface.
Section \ref{s6} investigates joint blow-ups, establishes a parabolic Liouville theorem, and characterizes the common
regular set $\mathcal R$. Finally, Section \ref{s10} proves the local backward coincidence of the joint and
nutrient interfaces and establishes the regularity of the joint free boundary at points of the common regular set $\mathcal R$.

\section{Preliminaries}\label{s2}

In this section, we introduce the relevant notation, specify the notion of weak solution, and describe the intrinsic scaling of the system.

\subsection*{The parabolic setting}
Throughout this paper, we assume $n\ge2$. We write $w_t:=\partial_t w$, and $(u,v)\ge0$
means $u,v\ge0$. Let $B_r(x)$ denote the open ball of radius $r$
centered at $x$, and set $B_r:=B_r(0)$. For $z_0=(x_0,t_0)$ we use
$$
\begin{aligned}
Q_r^-(z_0)&:=B_r(x_0)\times(t_0-r^2,t_0],\\
Q_r^+(z_0)&:=B_r(x_0)\times[t_0,t_0+r^2),\\
Q_r(z_0)&:=B_r(x_0)\times(t_0-r^2,t_0+r^2).
\end{aligned}
$$
We write $Q_r^-:=Q_r^-(0,0)$, $Q_r^+:=Q_r^+(0,0)$,
and $Q_r:=Q_r(0,0)$. Thus, $Q_r$ always denotes the
\emph{two-sided} cylinder, while $Q_r^-$ and $Q_r^+$
denote the backward and forward cylinders, respectively.
Throughout, the system is posed in $Q_1=B_1\times(-1,1)$.
For the backward cylinder, we set
$$
\partial_pQ_r^-:=(\partial B_r\times(-r^2,0])
\cup(B_r\times\{-r^2\}),
\qquad
\overline{Q_r^-(z_0)}:=\overline{B_r(x_0)}\times[t_0-r^2,t_0].
$$
The notation $\partial_pQ_r^-$ at a general center $z_0$
means the corresponding translated parabolic boundary. To avoid conflating past accessibility with the ordinary topology,
we write $\partial^-\{w>0\}$ for the \emph{backward free boundary}
of a continuous nonnegative function $w$:
$$
\partial^-\{w>0\}:=
\bigl\{z_0:w(z_0)=0,\quad
Q_r^-(z_0)\cap\{w>0\}\ne\varnothing
\text{ for every sufficiently small }r>0\bigr\}.
$$
We set
$$
\Gamma_u:=\partial^-\{u>0\},\qquad
\Gamma_v:=\partial^-\{v>0\},\qquad
\Gamma:=\partial^-\{|(u,v)|>0\}.
$$
We use $\partial$ for ordinary space-time boundaries. The \emph{backward closure} of the set $E$ is defined by
$$
\operatorname{cl}^-E
:=\{z_0:Q_r^-(z_0)\cap E\ne\varnothing
\text{ for all sufficiently small }r>0\}.
$$
In particular, for a continuous nonnegative function,
$\partial^-\{w>0\}=\{w=0\}\cap\operatorname{cl}^-\{w>0\}$.
The sets above are contained in the corresponding ordinary free boundaries, but the converse is not true in general. At nutrient regular points, the ordinary interface is a two-sided graph, and its phase separation identifies its backward trace with $\Gamma_v$. Compactness and local uniform convergence are understood in the ordinary space-time topology. Finally, for a measurable subset $E$ of $\R^d$, $|E|$ denotes its $d$-dimensional Lebesgue measure. Distances used for localization are taken with respect to the parabolic metric
$$
d_p((x,t),(y,s)):=\max\{|x-y|,|t-s|^{1/2}\}.
$$

\subsection*{Weak solutions}
Solutions of the system
\begin{equation}\label{NewSystem}
\left\{
\begin{array}{rcl}
\Delta u-u_t & = & v^\gamma \quad \text{in}\ Q_1\\
\Delta v-v_t & = & v^\theta \quad \, \text{in}\ Q_1
\end{array}
\right.
\end{equation}
are understood in the weak sense. More precisely, we consider functions 
$$
u,v\in C_{\mathrm{loc}}((-1,1);L^2_{\mathrm{loc}}(B_1))
\cap L^2_{\mathrm{loc}}((-1,1);H^1_{\mathrm{loc}}(B_1)),
$$
such that, for $w=u,v$, the equation
$$
\partial_t w-\Delta w=-F_w,\qquad
F_u=v^\gamma,\quad F_v=v^\theta,
$$
holds in the sense of distributions in $Q_1$.
Equivalently, for every $\varphi\in C_c^\infty(Q_1)$,
$$
\iint_{Q_1}\left(-w\,\partial_t\varphi
+D w\cdot D\varphi+F_w\varphi\right)\,dx\,dt=0.
$$
The triangular structure allows one first to solve the scalar nutrient equation and then the linear equation for the density defect. Nonnegativity of the density defect requires compatible data and is assumed throughout. Since the right-hand sides are bounded, standard interior parabolic $L^p$ estimates and embedding imply that both
components have locally H\"older-continuous representatives (see, for example, \cite{LSU68, RTU19}). In particular, their values and the free boundaries are well defined pointwise.

\subsection*{The intrinsic scaling} 
We conclude this section by analyzing the homogeneity of the system, which is determined by its invariance under parabolic scaling. More precisely, for $r>0$ and $z_0:=(x_0,t_0)$, we define
\begin{equation}\label{blow-ups}
    u_{r,z_0}(x,t):=\frac{1}{r^\alpha}u(x_0+rx,t_0+r^2t)
    \quad \mbox{and} \quad 
    v_{r,z_0}(x,t):=\frac{1}{r^\beta}v(x_0+rx,t_0+r^2t).
\end{equation}
Then, the second equation is invariant provided  
\begin{equation}\label{definition of beta}
    \beta=\frac{2}{1-\theta}.
\end{equation}
The invariance of the first equation then dictates
\begin{equation}\label{definition of alpha}
\alpha=\frac{2(1-\theta+\gamma)}{1-\theta}.
\end{equation}
In particular, $\alpha=\beta$ if and only if $\gamma=\theta$. We introduce the intrinsic norm
\begin{equation}\label{intrinsic norm}
    |(u,v)|:=u^{\frac{1}{\alpha}}+v^{\frac{1}{\beta}}.
\end{equation}

\section{Coupling of free boundaries}\label{s4}

In this section, we derive coupling properties for the nodal and positivity sets of $u$ and $v$, as well as for the associated free boundaries, leading to several geometric properties of the joint free boundary. 
\begin{proposition}\label{coupling lemma}
If $(u,v)\ge0$ is a solution of \eqref{NewSystem}, then
$$
\operatorname{int} \{u=0\}\subset \operatorname{int} \{v=0\},
\quad  
 \Gamma_u\subset\operatorname{cl}^-\{v>0\}, \quad
 \mbox{and}
\quad \Gamma_v \subset \operatorname{cl}^-\{u>0\}.
$$
\end{proposition}

\begin{proof} 
If $z_0\in\operatorname{int}\{u=0\}$ in the ordinary
space-time topology, then $u\equiv0$ in a two-sided cylinder
$Q_\rho(z_0)$ for some $\rho>0$. The first equation gives
$v^\gamma=0$ there, hence $v\equiv0$ in $Q_\rho(z_0)$.
This proves the first inclusion.

Assume now that the second inclusion fails, i.e., there exists $(x_0,t_0)\in\Gamma_u$ such that
$$
(x_0,t_0)\notin\operatorname{cl}^-\{v>0\}.
$$
Therefore, for $\rho>0$ small,
$$
Q_\rho^-(x_0,t_0)\cap\{v>0\}=\emptyset,
$$
i.e., $v\equiv0$ in $Q_\rho^-(x_0,t_0)$. Thus, 
$$
\Delta u-u_t=0 \quad \mbox{in} \ Q_\rho^-(x_0,t_0).
$$
On the other hand, as $(x_0,t_0)\in\Gamma_u$, then $u(x_0,t_0)=0$. The parabolic strong maximum principle implies that $u\equiv0$, which contradicts $(x_0,t_0)\in\Gamma_u$.

Similarly, to conclude the third inclusion, assume that for some $z_0\in\Gamma_v$ one has $u\equiv0$ in some $Q_\rho^-(z_0)$. Then the first equation forces $v\equiv0$ there, contradicting the assumption that $z_0\in\Gamma_v$.
\end{proof}

\begin{proposition}\label{coupling prop}
    If $(u,v)\ge 0$ is a solution of \eqref{NewSystem}, then
    $$
    \Gamma=\Gamma_u\cap\{v=0\}=\Gamma_v\cap\{u=0\}.
    $$
\end{proposition}

\begin{proof}
    Let $z_0=(x_0,t_0)\in\Gamma$. By definition,
    $$
    u(z_0)=v(z_0)=0.
    $$
    We claim that $z_0$ belongs to both $\Gamma_u$ and $\Gamma_v$.
    Indeed, the failure of $z_0\in\Gamma_u$ would imply that
    $u$ vanishes in a sufficiently small backward cylinder centered at $z_0$. The first equation then yields $v\equiv0$ there, contradicting
    $z_0\in\Gamma$.

    Similarly, the failure of $z_0\in\Gamma_v$ would imply that $v$ vanishes locally. Consequently, $u$ is a nonnegative caloric function attaining its minimum at $z_0$. By the parabolic strong maximum principle, $u$ vanishes in a sufficiently small backward cylinder, again contradicting $z_0\in\Gamma$. We have therefore established
    $$
    \Gamma\subset\Gamma_u\cap\{v=0\}\quad\text{and}\quad\Gamma\subset\Gamma_v\cap\{u=0\}.
    $$
    The reverse inclusions follow directly from the definitions. Every point of $\Gamma_u\cap\{v=0\}$ is a common zero of $u$ and $v$ approached by points where $|(u,v)|\geq u^{\frac{1}{\alpha}}>0$, and hence belongs to $\Gamma$. The same reasoning applies to $\Gamma_v\cap\{u=0\}$, completing the proof.
\end{proof}
\begin{corollary}\label{coupling corollary}
If $(u,v)\geq0$ is a solution of \eqref{NewSystem}, then
$$
\Gamma=\Gamma_u\cap\Gamma_v.
$$
In particular,
$$
\Gamma\subset
\operatorname{cl}^-\{u>0\}
\cap
\operatorname{cl}^-\{v>0\}.
$$
\end{corollary}

\section{Sharp growth dichotomy}\label{s3}

In this section, we study the local behavior of nonnegative bounded solutions of \eqref{NewSystem}, and the geometric structure of the associated joint free boundary. We begin by establishing optimal growth estimates for both components near the free boundary. We then explore the coupling between the two equations, establishing
relations between the zero sets, positivity sets, and corresponding free boundaries of $u$ and $v$.

\subsection*{Optimal growth estimates} 
We begin by establishing an improvement of flatness mechanism for solutions vanishing at a dead-core point, showing that, after a suitable
normalization, sufficiently flat solutions become even flatter at smaller scales.

\begin{lemma}\label{flattening lemma}
    For any given $\mu>0$, there exists $\kappa>0$, depending only on $\mu$, $\gamma$, $\theta$ and $n$, such that if $(u,v)\ge0$ solves
    \begin{equation*}
    \left\{
    \begin{array}{rcr}
    \Delta u-u_t& = &v^\gamma \quad \text{in}\ Q_1^-\\
    \Delta v-v_t& = &\delta v^\theta \quad \text{in}\ Q_1^-,
    \end{array}
    \right.
\end{equation*}
    for $\delta\in(0,\kappa)$, and
    \begin{equation}\label{conditions}
        \|u\|_{L^\infty(Q_1^-)}\le 1,\quad \|v\|_{L^\infty(Q_1^-)} \le 1,\quad u(0,0)=v(0,0)=0,
    \end{equation}
    then
    $$
    \sup_{Q_{1/2}^-}|(u,v)|\le\mu.
    $$
\end{lemma}
\begin{proof}
   Indeed, if not, then for a certain $\mu_0>0$, there exist sequences $u_i, v_i$ satisfying \eqref{conditions}, such that in $Q_1^-$ one has
   \begin{equation*}
    \left\{
    \begin{array}{rcr}
    \Delta u_i-(u_i)_t& = &v_i^\gamma\\[1mm]
    \Delta v_i-(v_i)_t& = &i^{-1}v_i^\theta,
    \end{array}
    \right.
\end{equation*}
but 
\begin{equation}\label{contradictory assumption}
    \sup_{Q_{1/2}^-}|(u_i,v_i)|>\mu_0.
\end{equation}
By classical parabolic regularity theory, $\{(u_i,v_i)\}_{i\ge1}$ is uniformly bounded and equicontinuous in compact subsets of $Q_1^-$. Therefore, up to a subsequence, $u_i,v_i$ converge locally uniformly in $Q_{2/3}^-$ to $u_\infty$ and $v_\infty$, respectively, as $i\to\infty$. Clearly, $u_\infty,v_\infty$ satisfy \eqref{conditions}, and $v_\infty$ is a caloric function in $Q_{2/3}^-$, i.e.,
$$
\Delta v_\infty-(v_\infty)_t=0.
$$
Since $v_\infty\ge0$ attains its minimum at the interior point $(0,0)$, the parabolic strong maximum principle implies $v_\infty\equiv0$ in $Q_{2/3}^-$. The latter, in turn, yields
$$
\Delta u_\infty-(u_\infty)_t=0.
$$
Using the strong maximum principle again, we get $u_\infty\equiv0$ in $Q_{2/3}^-$, which contradicts \eqref{contradictory assumption}.
\end{proof}
Iterating, we derive the optimal growth estimate for nonnegative bounded solutions of \eqref{NewSystem}.

\begin{theorem}\label{optimal growth theorem}
    If $(u,v)\ge0$ is a bounded solution of \eqref{NewSystem}, then there exists a constant $C>0$, depending only on $\gamma$, $\theta$, $\|u\|_\infty$, $\|v\|_\infty$ and $n$, such that, for each $(x_0,t_0)\in\Gamma\cap Q_{1/2}^-$, there holds
    $$
    \sup_{Q_r^-(x_0,t_0)}|(u,v)|\le Cr,
    $$
    for $Q_r^-(x_0,t_0)\subset Q_1^-$. In particular, in $Q_r^-(x_0,t_0)$ one has
    $$
    u\le Cr^{\frac{2(1-\theta+\gamma)}{1-\theta}}\qquad\text{and}\qquad v\le Cr^{\frac{2}{1-\theta}}.
    $$
\end{theorem}
\begin{proof}
Without loss of generality, we assume $(x_0,t_0)=(0,0)$. 
Take $\mu:=\frac{1}{2}$ and let $\kappa>0$ be as in Lemma \ref{flattening lemma}. For
    $$
    A>\max\left\{1,\|u\|_\infty^{\frac{1}{\gamma}}, \|v\|_\infty, \kappa^{\frac{1}{\theta-1}}\right\},
    $$
    define
    $$
    u_1:=\frac{u}{A^\gamma}\qquad\text{and}\qquad v_1:=\frac{v}{A}.
    $$
    Then $(u_1,v_1)$ satisfies \eqref{conditions} and
    \begin{equation}\label{iteration}
    \left\{
    \begin{array}{rcr}
    \Delta u_1-(u_1)_t& = &v_1^\gamma\\[1mm]
    \Delta v_1-(v_1)_t& = &\delta v_1^\theta,
    \end{array}
    \right.
\end{equation}
with $\delta:=A^{\theta-1}<\kappa$. By Lemma \ref{flattening lemma},
\begin{equation}\label{the first step}
    \sup_{Q_{1/2}^-}|(u_1,v_1)|\le\mu.
\end{equation}
We now iterate. For any $i\in\mathbb{N}$, set
\begin{align*}
    u_i(x,t)&:=2^{(i-1)\alpha}u_1(2^{1-i}x,2^{2(1-i)}t),\\
    v_i(x,t)&:=2^{(i-1)\beta}v_1(2^{1-i}x,2^{2(1-i)}t),
\end{align*}
where $\alpha$ and $\beta$ are defined by \eqref{definition of alpha} and \eqref{definition of beta}, respectively. Note that $(u_i,v_i)$ solves \eqref{iteration} and $u_i(0,0)=v_i(0,0)=0$. We aim to prove that
\begin{equation}\label{induction}
    \sup_{Q_{1/2}^-}|(u_i,v_i)|\le\mu,\quad i\in\mathbb{N}.
\end{equation}
Indeed, for $i=1$, this is \eqref{the first step}. We argue by induction, supposing \eqref{induction} is true for $i$, and deducing that the same holds for $i+1$. By \eqref{induction}, 
$$
u_i\le\mu^\alpha\qquad \text{and}\qquad v_i\le\mu^\beta\quad \text{in}\ Q_{1/2}^-.
$$
Now, if $(x,t)\in Q_1^-$, then $(\frac{x}{2},\frac{t}{4})\in Q_{1/2}^-$, therefore
$$
u_{i+1}(x,t) = 2^\alpha u_i\left(\frac{x}{2},\frac{t}{4}\right)\le2^\alpha\mu^\alpha=1.
$$
Similarly,
$$
v_{i+1}(x,t)\le1.
$$
Thus, $(u_{i+1},v_{i+1})$ satisfies \eqref{conditions}. Applying Lemma \ref{flattening lemma} to $(u_{i+1},v_{i+1})$, we obtain \eqref{induction} for $i+1$. Scaling back, for any $i\in\mathbb{N}$, we get
\begin{equation}\label{scaling back}
    \sup_{Q_{2^{-i}}^-}|(u_1,v_1)|\le\mu2^{-(i-1)}.
\end{equation}
Let now $0<r\le\frac{1}{2}$ and choose $i\in\mathbb{N}$ such that
$$
2^{-(i+1)}<r\le2^{-i}.
$$
Then $Q_r^-\subset Q_{2^{-i}}^-$ and since $2^{-(i-1)}<4r$, from \eqref{scaling back} one has
$$
\sup_{Q_r^-}|(u_1,v_1)|\le4\mu r=2r.
$$
The latter, recalling the definition of $(u_1,v_1)$, for any $r\in(0,\frac{1}{2}]$, yields
$$
\sup_{Q_r^-}|(u,v)|\le Cr,
$$
where
$$
C:=2\max\left\{A^{\frac{\gamma}{\alpha}}, A^\frac{1}{\beta}\right\}.
$$
\end{proof}
\begin{remark}[Localization of backward growth]
\label{backward growth localization}
The estimate of Theorem~\ref{optimal growth theorem}, stated for points
in $\Gamma\cap Q_{1/2}^-$, is local: translation, parabolic
rescaling, and the same normalization argument show that, for every
compact $K\Subset Q_1$, there exist $C_K$ and $r_K>0$ such that
$$
\sup_{Q_r^-(z_0)}u\le C_Kr^\alpha,
\qquad
\sup_{Q_r^-(z_0)}v\le C_Kr^\beta,
\qquad z_0\in\Gamma\cap K,\quad r\leq r_K.
$$
This localized form is used below, now including joint points with positive time coordinate.
\end{remark}

\begin{proposition}[Two-sided optimal growth]\label{two sided growth}
If $(u,v)\ge0$ is a bounded solution of \eqref{NewSystem}, then, for every compact set $K\Subset Q_1$, there exist $C>0$ and $r_K>0$, depending only on $K$, $n$, $\gamma$, $\theta$, $\|u\|_{L^\infty(Q_1)}$ and $\|v\|_{L^\infty(Q_1)}$, such that, for every $z_0\in\Gamma\cap K$ and $0<r<r_K$,
\begin{equation}\label{two sided growth bound}
\sup_{Q_r(z_0)}u\le Cr^\alpha,\qquad\sup_{Q_r(z_0)}v\le Cr^\beta.
\end{equation}
\end{proposition}

\begin{proof}
The corresponding estimates in $Q_r^-(z_0)$ follow from
Remark~\ref{backward growth localization}.
Fix $z_0=(x_0,t_0)\in\Gamma\cap K$.
In particular, evaluating the backward estimates at time $t_0$
and choosing $r$ slightly larger than $|x-x_0|$ gives
\begin{equation}\label{initial trace growth}
u(x,t_0)\le C_0|x-x_0|^\alpha,
\qquad
v(x,t_0)\le C_0|x-x_0|^\beta,
\end{equation}
for $x$ in a fixed ball $B_{R_0}(x_0)$, with
$R_0$ chosen smaller than the localization radius in
Remark~\ref{backward growth localization}.
Choose $R_0>0$, depending only on $\dist_p(K,\partial Q_1)$,
so that $B_{R_0}(x_0)\times[t_0,t_0+R_0^2]$ lies in $Q_1$
for every $z_0\in K$. For $p>2$, $M>0$ and $t\ge t_0$ set
$$
\Phi_p(x,t):=M\bigl(|x-x_0|^2+\Lambda_p(t-t_0)\bigr)^{p/2},
\qquad \Lambda_p:=2(n+p-2).
$$
Writing $S:=|x-x_0|^2+\Lambda_p(t-t_0)$, we compute
$$
(\partial_t-\Delta)\Phi_p
=\frac{Mp}{2}S^{p/2-1}
\left(\Lambda_p-2n-2(p-2)\frac{|x-x_0|^2}{S}\right)\ge0,
$$
with the inequality understood by continuity where $S=0$.
Choose $M\ge C_0$ such that
$MR_0^p\ge\|u\|_{L^\infty(Q_1)}$ for $p=\alpha$, and
$MR_0^p\ge\|v\|_{L^\infty(Q_1)}$ for $p=\beta$.
The constants $M$ and $R_0$ are independent of $z_0\in K$.
By \eqref{initial trace growth}, the corresponding barrier dominates
its component on the initial time slice, and by the choice of $M$, it also dominates on the lateral boundary $|x-x_0|=R_0$.
Since
$$
(\partial_t-\Delta)u=-v^\gamma\le0,
\qquad
(\partial_t-\Delta)v=-v^\theta\le0,
$$
the parabolic comparison principle gives
$$
u(x,t)\le\Phi_\alpha(x,t),\qquad
v(x,t)\le\Phi_\beta(x,t),
$$
for $|x-x_0|<R_0$ and $t_0\le t\le t_0+R_0^2$. For $0<r<R_0$, this yields the required estimates in $Q_r^+(z_0)$. Combining them with the backward estimates gives \eqref{two sided growth bound}.
\end{proof}

\begin{remark}\label{r3.1}
    Theorem \ref{optimal growth theorem} is optimal. Indeed, the pair
    $$
    (u,v):=\left( a(x_1)_+^\alpha,b(x_1)_+^\beta \right),
    $$
    where $\alpha$ and $\beta$ are defined respectively by \eqref{definition of alpha} and \eqref{definition of beta}, and 
    $$
    a:=\frac{b^\gamma}{\alpha(\alpha-1)},\qquad b:=[\beta(\beta-1)]^{\frac{1}{\theta-1}},
    $$
    solves \eqref{NewSystem}.  At any $(x_0,t_0) \in \Gamma=\{x_1=0\}$,
    one has
    $$
    \sup_{Q_r^-(x_0,t_0)}u=ar^\alpha,\qquad\sup_{Q_r^-(x_0,t_0)}v=br^\beta.
    $$
\end{remark}

\subsection*{Sharp growth dichotomy}

We conclude this section by examining how the growth of the density defect changes along its free boundary $\Gamma_u$. The coupling between the two equations gives rise to a quadratic--superquadratic dichotomy: at points of $\Gamma_u\setminus\Gamma$, where the nutrient concentration is strictly positive, the function $u$ exhibits nontrivial quadratic growth. In contrast, at joint free boundary points, the vanishing nutrient concentration leads to the higher intrinsic growth of order $r^\alpha$, where $\alpha=2+\gamma\beta>2$. The following result captures these two regimes in a unified growth estimate. 
\begin{corollary}
\label{c4.6}
If $(u,v)\geq0$ is a bounded solution of
\eqref{NewSystem} in $Q_1$, then there exist universal constants $C>0$ and $r_0>0$, such that, for each $z_0\in \Gamma_u$, there holds
\begin{equation}\label{dichotomyestim}
\sup_{Q_r(z_0)}u
\leq
Cr^2\varphi(z_0,r),
\end{equation}
for $0<r<r_0$, where
$$
\varphi(z,r):=
\left\{
\begin{array}{rl}
   v^\gamma(z)+r^\gamma,  & \mbox{if } z\in \Gamma_u\setminus \Gamma\\[1mm]
   r^{\gamma\beta},  & \mbox{if } z\in \Gamma.
\end{array}
\right.
$$
\end{corollary}

\begin{proof}
We first establish the required interior
regularity. Since $v^\theta$ is bounded, standard interior parabolic $W^{2,1}_p$ estimates, followed by a parabolic Sobolev embedding with $p>n+2$, imply that $v$ is locally Lipschitz continuous with respect to the parabolic distance. More precisely, for every $K'\Subset Q_1$,
$$
|v(x,t)-v(y,s)|
\leq C_{K'}
\left(
|x-y|+|t-s|^{1/2}
\right),
$$
for $(x,t),(y,s)\in K'$ sufficiently close. Since $0<\gamma<1$ and $v\geq0$, we have
$$
|v(x,t)^\gamma-v(y,s)^\gamma|
\leq |v(x,t)-v(y,s)|^\gamma.
$$
Consequently, $v\in C_{\loc}^{\gamma,\gamma/2}(Q_1)$. Applying the interior parabolic Schauder estimates to the first equation in \eqref{NewSystem}, we obtain that $u$ is locally $C^{2+\gamma,1+\gamma/2}(Q_1)$,
with universal estimates. Fix now $z_0=(x_0,t_0)\in\Gamma_u$. Since $z_0$ is an interior minimum of $u$, one has
\begin{equation}\label{hessian}
u(z_0)=0, \quad D_xu(z_0)=0,\quad
u_t(z_0)=0,\quad \mbox{and} \quad A_{z_0}:=D_x^2u(z_0)\geq0.
\end{equation}
Suppose $z_0\in\Gamma_u\setminus\Gamma=\Gamma_u\setminus\Gamma_v$. The last equality follows from Corollary \ref{coupling corollary}. Since $v(z_0)>0$,
$$
\operatorname{tr}A_{z_0}
=v^\gamma(z_0)>0,
$$
and hence
$$
0\le\langle A_{z_0}\xi,\xi\rangle\le\operatorname{tr}A_{z_0}\,|\xi|^2=v^\gamma(z_0)|\xi|^2,
$$
for every $\xi\in\R^n$. Furthermore, the parabolic Taylor expansion of $u$ at $z_0$ yields
$$
u(x,t)=\frac12\langle A_{z_0}(x-x_0),x-x_0\rangle+R_{z_0}(x,t),
$$
where
$$
|R_{z_0}(x,t)|\le C\left(|x-x_0|+|t-t_0|^{1/2}\right)^{2+\gamma}.
$$
The latter, in view of \eqref{hessian}, yields
$$
u(x,t)\le\frac12 v^\gamma(z_0)|x-x_0|^2+C\left(|x-x_0|+|t-t_0|^{1/2}\right)^{2+\gamma}.
$$
Therefore, 
$$
\sup_{Q_r(z_0)}u\le C\left(v^\gamma(z_0) r^2+r^{2+\gamma}\right).
$$
If $z_0\in\Gamma$, then \eqref{dichotomyestim} follows from Proposition \ref{two sided growth}.
\end{proof}
The next result shows that, as points of $\Gamma_u\setminus\Gamma$ approach the joint free boundary $\Gamma$, quadratic profiles degenerate at a quantitative rate.

\begin{corollary}\label{c4.7}
If $(u,v)\geq0$ is a bounded solution of \eqref{NewSystem}, then 
$$
\Gamma=\Gamma_u\cap\{D_x^2u=0\}.
$$
Moreover, for every compact set $K\Subset Q_1$, there exist $C>0$ and $\delta>0$ such that
$$
\|D_x^2u(z)\|\le C\,\left[\dist_p(z,\Gamma)\right]^{\gamma\beta},
$$
whenever $z\in(\Gamma_u\setminus\Gamma)\cap K$ and
$\dist_p(z,\Gamma)<\delta$.
\end{corollary}

\begin{proof}
As in the proof of Corollary \ref{c4.6}, for each $z_0 \in \Gamma_u$, we have
\begin{equation}\label{traceeq}
\operatorname{tr}D_x^2u(z_0)=v^\gamma(z_0).
\end{equation}
Since $D_x^2u(z_0)$ is positive semidefinite, then
$$
D_x^2u(z_0)=0 \quad\Longleftrightarrow\quad v(z_0)=0.
$$
Recalling $\Gamma=\Gamma_u\cap\{v=0\}$ and Proposition \ref{coupling prop}, we obtain
$$
\Gamma=\Gamma_u\cap\{D_x^2u=0\}.
$$
If $z_0\in\Gamma_u\setminus\Gamma$ and $\lambda_{\max}(A)$ is the maximal eigenvalue of the matrix $A$, then from \eqref{traceeq} we get
$$
\|D_x^2u(z_0)\|=\lambda_{\max}(D_x^2u(z_0))\le\operatorname{tr}D_x^2u(z_0)=v^\gamma(z_0).
$$
Let $d:=\dist_p(z_0,\Gamma)$ and assume that $d>0$ is sufficiently small. Choose $\zeta\in\Gamma$ such that
$$
d_p(z_0,\zeta)<2d, 
$$
and so $z_0\in Q_{3d}(\zeta)$. Hence,
$$
v(z_0)\le\sup_{Q_{3d}(\zeta)}v\le C d^\beta,
$$
and therefore,
$$
\|D_x^2u(z_0)\|_{\mathrm{op}}\le v^\gamma(z_0)\le Cd^{\gamma\beta}.
$$
\end{proof}

\begin{center}
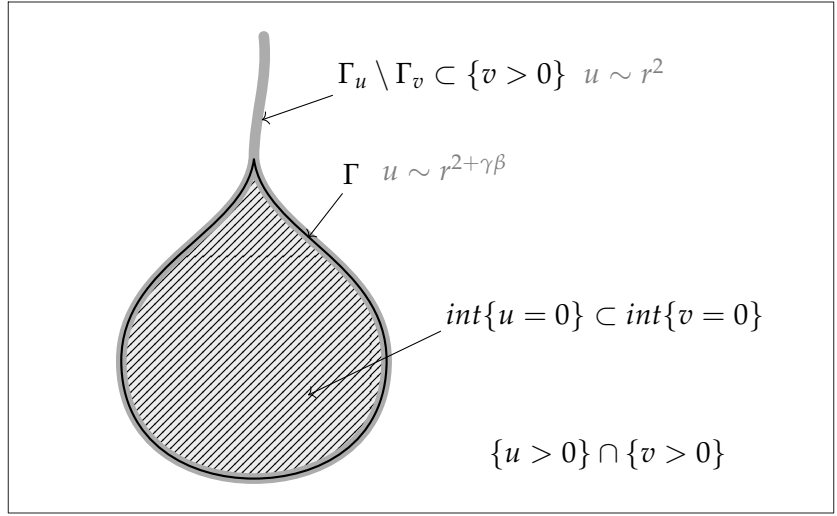
\begin{figure}[h!]
\begin{tikzpicture}[
    scale=1.3,
    line cap=round,
    line join=round
]

\def\cebola{
    (0,1.6)
    .. controls (-0.10,0.85) and (-1.35,0.55)
       .. (-1.35,-0.45)
    .. controls (-1.35,-1.30) and (-0.65,-1.65)
       .. (0,-1.65)
    .. controls (0.65,-1.65) and (1.35,-1.30)
       .. (1.35,-0.45)
    .. controls (1.35,0.55) and (0.10,0.85)
       .. (0,1.6)
    -- cycle
}

\def\linha{
    (0,1.6)
    .. controls (0,2.05) and (0.15,2.40)
       .. (0.10,2.85)
}

\fill[gray!15] \cebola;
\fill[
    pattern=north east lines,
    pattern color=black
] \cebola;

\draw[gray!65, line width=4pt] \cebola;
\draw[gray!65, line width=4pt] \linha;

\draw[black, line width=0.8pt] \cebola;

\node[anchor=west] at (3.25,2.5)
    {\color{gray}$u\sim r^2$};
\node[anchor=west] at (0.75,2.45)
    {$\Gamma_u\setminus \Gamma_v  \subset \{v>0\}$};
\draw[black, thin, ->]
    (0.8,2.25) -- (0.08,2.0);
    
\node[anchor=west] at (1.85,0.0)
{$int\{u=0\}\subset int\{v=0\}$ };
\draw[black, thin, ->]
    (1.9,-0.15) -- (0.55,-0.8);

\node[anchor=west] at (0.8,1.45)
    {$\Gamma$};
\node[anchor=west] at (1.2,1.52)
    {\color{gray}$u\sim r^{2+\gamma\beta}$};
\draw[black, thin, ->]
    (0.9,1.25) -- (0.55,0.8);

\node at (3.6,-1.4) {$\{u>0\}\cap \{v>0\}$};
%

\draw[gray!35!black] 
(-2.5,3.2) -- (-2.5,-2) -- (5.9,-2) -- (5.9,3.2) -- (-2.5,3.2);
\end{tikzpicture}
\caption{Illustration of the growth dichotomy along the saturation interface $\Gamma_u$ (denoted by the thicker gray boundary), distinguishing quadratic growth on $\Gamma_u\setminus\Gamma$ from superquadratic growth on the joint free boundary $\Gamma$ (the black line).}
\end{figure}
\end{center}

\section{Nondegeneracy and consequences}\label{s5}

The second equation in \eqref{NewSystem} yields an analog of the weak comparison theorem in \cite{AT24}, suggesting that a corresponding nondegeneracy result and its consequences should also hold.

\subsection{An auxiliary barrier function}
We start by defining the barrier function
\begin{equation}\label{the barrier function}
    \psi(x,t):=m\left(\sqrt{|x|^2-t}-\frac{1}{4}\right)_+^\beta, \quad (x,t)\in Q_1^-,
\end{equation}
where $m>0$ is a small constant to be chosen later, and $\beta$ is defined by \eqref{definition of beta}. Since $\psi$ vanishes in the region $\sqrt{|x|^2-t}\le\frac{1}{4}$, there is a universal $\sigma>0$ such that
\begin{equation}\label{vanishing psi}
    \psi\equiv 0 \quad \text{in} \ Q_\sigma^-.
\end{equation}
Setting
$$
s:=\sqrt{|x|^2-t},\qquad h:=s-\frac{1}{4},
$$
we have $\psi=mh_+^\beta$. In the region where $h>0$, by direct calculations, we obtain
$$
\Delta\psi=m\beta h^{\beta-1}\Delta s+m\beta(\beta-1)h^{\beta-2}|Ds|^2
$$
and
$$
\psi_t=-\frac{m\beta}{2s}h^{\beta-1}.
$$
Therefore,
$$
\Delta\psi-\psi_t\le|\Delta\psi|+|\psi_t|\le Cmh^{\beta\theta},
$$
with $C:=\beta(n+\beta-\frac{1}{2})>0$. Since $\theta<1$, we can always choose $m>0$ small enough so that
$$
Cm\le m^\theta,
$$
and hence,
$$
\Delta\psi-\psi_t\le\psi^\theta,
$$
i.e., $\psi$ is a supersolution of the second equation in \eqref{NewSystem}. Observe, additionally, that on the parabolic boundary of $Q_1^-$, 
$$
\partial_pQ_1^-=\left(\partial B_1\times(-1,0]\right)\cup\left(B_1\times\{-1\}\right),
$$
one has
$$
\sqrt{|x|^2-t}\ge1.
$$
Hence,
\begin{equation}\label{lower estimate for barrier}
    \psi\ge m\left(\frac{3}{4}\right)^\beta =: c_* \quad \text{on} \ \partial_pQ_1^-.
\end{equation}

\subsection{Nondegeneracy estimates}
We already know that the intrinsic norm of any nonnegative pair of solutions of \eqref{NewSystem} grows at most linearly. We now show that its $\sup$ grows at least linearly. 

\begin{theorem}\label{non-degeneracy}
    If $(u,v)\ge0$ is a bounded solution of \eqref{NewSystem}, and $z_0=(x_0,t_0)\in\Gamma$, then there exists a constant $c>0$, depending only on $\gamma$, $\theta$ and $n$, such that
    $$
    \sup_{\overline{Q_r^-(x_0,t_0)}}|(u,v)|\ge cr,
    $$
    whenever $r>0$ is such that $\overline{Q_r^-(z_0)}\subset Q_1$.
\end{theorem}
\begin{proof} 
    We first prove that
    \begin{equation}\label{v nondegeneracy}
        \sup_{\overline{Q_r^-(x_0,t_0)}}v\ge c_*r^\beta,
    \end{equation}
    where, as before, $\beta$ is defined by \eqref{definition of beta}, and $c_*$ is defined in \eqref{lower estimate for barrier}. After scaling and translation, it is enough to prove \eqref{v nondegeneracy} for $(x_0,t_0)=(0,0)$ and $r=1$, i.e., we aim to prove
    \begin{equation}\label{scaled v nondegeneracy}
        \sup_{\overline{Q_1^-}}v\ge c_*.
    \end{equation}
    If \eqref{scaled v nondegeneracy} fails, i.e.,
    $$
    \sup_{\overline{Q_1^-}}v<c_*,
    $$
    then
    $$
    v\le\psi \quad \text{on} \ \partial_p Q_1^-,
    $$
    where $\psi\ge0$ is the barrier function defined by \eqref{the barrier function}. Since $v$ is a solution and $\psi$ is a supersolution of the second equation in \eqref{NewSystem}, the parabolic comparison principle yields
    $$
    v\le\psi \quad \text{in} \ Q_1^-.
    $$
    Then, thanks to \eqref{vanishing psi},
    $$
    v\equiv0 \quad \text{in} \ Q_\sigma^-.
    $$
    However, this contradicts Proposition \ref{coupling lemma}, which yields that $(0,0)\in\operatorname{cl}^-\{v>0\}$. Thus \eqref{scaled v nondegeneracy}, and hence \eqref{v nondegeneracy}, holds. It remains to observe that
    $$
    |(u,v)|\ge v^{\frac{1}{\beta}},
    $$
    and the result follows with $c:=c_*^{\frac{1}{\beta}}$.
\end{proof}

\subsection{Geometry of the joint free boundary}
With the growth and nondegeneracy estimates, we deduce the quantitative thickness of the tumor boundary as in \cite[Theorem 6.2]{RTU19}. We recall the definition of space-time porosity.
\begin{definition}
The set $E\subset\R^n\times\R$ is called space-time porous if there exist constants $\delta\in(0,1)$ and $r_0>0$ such that, for every $(x_0,t_0)\in E$ and every $0<r<r_0$, there exists $(y,s)\in Q_r^-(x_0,t_0)$ for which
$$
Q_{\delta r}^-(y,s)\subset Q_r^-(x_0,t_0)\setminus E.
$$
\end{definition}
With respect to the parabolic metric $d_p$, Lebesgue measure is $(n+2)$-regular. A backward cylinder contains a parabolic metric ball of comparable radius. The standard dimension estimate for porous sets then yields
$$
\dim_pE\le n+2-c\delta^{n+2},
$$
for a constant $c>0$ depending only on the dimension (see \cite[Theorem 4.7]{JJKRRS10}). Here, $\dim_pE$ is the parabolic Hausdorff dimension of the set $E$. In particular, a space-time porous set has zero Lebesgue measure.

\begin{corollary}
If $(u,v)\ge0$ is a bounded solution of \eqref{NewSystem}, then for each compact set $K\Subset Q_1$, the joint free boundary $\Gamma\cap K$ is porous with respect to the parabolic metric, with a porosity constant depending only on $n$, $\gamma$, $\theta$, $\|u\|_\infty$, $\|v\|_\infty$, and $\dist_p(K,\partial Q_1)$.
\end{corollary}
\begin{proof}
    Take $(x_0,t_0)\in\Gamma\cap K$, and let $r>0$ be so small that $\overline{Q_r^-(x_0,t_0)}\subset Q_1$. By Theorem~\ref{non-degeneracy}, there exists $(y,s)\in\overline{Q_{r/2}^-(x_0,t_0)}$ such that
    $$
    v(y,s)\ge cr^\beta.
    $$
    Note that $Q_{r/4}^-(y,s)\subset Q_r^-(x_0,t_0)$. On the other hand, by Theorem~\ref{optimal growth theorem},
    $$
    v\le Cr^\beta \quad \text{in} \ Q_r^-(x_0,t_0).
    $$
    Observe now that
    $$
    \tilde{v}(x,t):=\frac{v(y+rx, s+r^2t)}{r^\beta}
    $$
    satisfies the second equation of \eqref{NewSystem} and
    $$
    0\le\tilde{v}\le C \quad \text{in} \ Q_{1/4}^-.
    $$
    Furthermore,
    $$
    \tilde{v}(0,0)\ge c,
    $$
    and by parabolic interior H\"older regularity for the heat equation (see, for example, \cite{RTU19}), there exists $\sigma\in(0,1)$ such that
    $$
    |\tilde{v}(x,t)-\tilde{v}(0,0)|\le\tilde{C}\left(|x|+|t|^{1/2}\right)^\sigma, \quad (x,t) \in Q_{1/8}^-,
    $$
    for some constant $\tilde{C}>0$. Choosing $\delta\in(0,1/8)$ so small that
    $$
    \tilde{C}(2\delta)^\sigma\le\frac{c}{2},
    $$
    for every $(x,t)\in Q_\delta^-$, one has
    $$
    \tilde{v}(x,t)\ge\tilde{v}(0,0)-|\tilde{v}(x,t)-\tilde{v}(0,0)|\ge c-\frac{c}{2}>0.
    $$
    Scaling back, we obtain
    $$
    v \ge \left( \frac{c}{2} \right)r^\beta>0 \quad \text{in} \ Q_{\delta r}^-(y,s)\subset Q_r^-(x_0,t_0),
    $$
    and so
\begin{equation}\label{porosity inclusion}
        Q_{\delta r}^-(y,s)\subset Q_r^-(x_0,t_0)\setminus\Gamma,
    \end{equation}    
which concludes the proof.
\end{proof}

As a consequence, we obtain the following positive nutrient density result, establishing positive uniform space-time density at every point on the joint boundary. Below, we write \(|Q_r^-|\) and \(|B_r|\) for the \((n+1)\)- and \(n\)-dimensional Lebesgue measures of the cylinder \(Q_r^-\) and the ball \(B_r\), respectively.

\begin{corollary}\label{positive density}
    If $(u,v)\ge0$ is a bounded solution of \eqref{NewSystem} in $Q_1$, and $K\Subset Q_1$, then there exist $c>0$ and $r_\star>0$ such that, for every $(x_0,t_0)\in\Gamma\cap K$ and $r\in(0,r_\star)$, one has
    $$
    \frac{|Q_r^-(x_0,t_0)\cap\{|(u,v)|>0\}|}{|Q_r^-(x_0,t_0)|}\ge c.
    $$
    In particular, in the present two-sided setting,
    $$
    \frac{|Q_r(x_0,t_0)\cap\{|(u,v)|>0\}|}
         {|Q_r(x_0,t_0)|}\ge\frac c2,
    $$
    whenever $Q_r(x_0,t_0)\subset Q_1$.
\end{corollary}

\begin{proof}
Indeed, by \eqref{v nondegeneracy} and \eqref{porosity inclusion}, since $\{v>0\}\subset\{|(u,v)|>0\}$, we have 
$$
Q_{\delta r}^-(y,s)\subset Q_r^-(x_0,t_0)\cap\{|(u,v)|>0\}.
$$
From this,
$$
|Q_r^-(x_0,t_0)\cap\{|(u,v)|>0\}|\ge|Q_{\delta r}^-(y,s)|,
$$
which implies
$$
|Q_r^-(x_0,t_0)\cap\{|(u,v)|>0\}| \ge |B_1|(\delta r)^{n+2} =\delta^{n+2}|Q_r^-(x_0,t_0)|.
$$
The backward estimate follows with $c:=\delta^{n+2}$. The two-sided estimate is immediate from $Q_r^-(z_0)\subset Q_r(z_0)$ and $|Q_r(z_0)|=2|Q_r^-(z_0)|$.
\end{proof}

\begin{remark}\label{r5.1}
The proof of Corollary \ref{positive density} yields a stronger density estimate for the nutrient positivity set. Precisely, there exists a universal constant
    $\delta\in(0,1)$ such that
    $$
        \frac{|Q_r^-(x_0,t_0)\cap\{v>0\}|}
        {|Q_r^-(x_0,t_0)|}
        \geq \delta^{n+2},
    $$
    for each $(x_0,t_0)\in\Gamma$ and small $r>0$.
\end{remark}

\section{Joint blow-up classification and rigidity}\label{s6}

In this section, we study the simultaneous blow-up behavior of solutions at the joint free boundary $\Gamma$. Building on the classification and uniqueness of nutrient blow-ups at regular free boundary points established in \cite{AFSW22}, we derive that the coupling structure of the system uniquely determines the blow-up of the density defect. Consequently, both components admit a unique joint blow-up, with a common positivity half-space. The converse also holds: a half-space blow-up of the density defect determines the nutrient profile through the two equations. This identifies a common regular set along the joint free boundary $\Gamma$.

Throughout the section, we denote
\begin{equation}\label{definition of H and c}
c_\beta :=\left[\beta(\beta-1)\right]^{\frac{1}{\theta-1}} \quad \mbox{and} \quad c_{\alpha,\beta}=\frac{c_{\beta}^\gamma}{\alpha(\alpha-1)}.
\end{equation}
We adopt a two-sided, subsequential notion of regularity consistent with the regular point framework of \cite[Definition 4.4]{AFSW22}. In the following definitions, the $u_{r,z_0}$ and $v_{r,z_0}$ are defined by
\eqref{blow-ups}. Precompactness at joint free boundary points follows from Proposition \ref{two sided growth}.

\begin{definition}\label{regular points}
A point $z_0\in\Gamma_v$ is called regular if there exist $\nu\in\mathbb S^{n-1}$ and a sequence $r_j\to0^+$ such that
$$
v_{r_j,z_0}(x,t)
\longrightarrow
c_\beta(x\cdot\nu)_+^\beta,
$$
locally uniformly in $\R^n\times\R$. The respective set of regular points is denoted by $\mathcal R_v$.
\end{definition}
\begin{definition}\label{regular points u}
A point $z_0\in\Gamma_u$ is called regular if there exist $A>0$, $\nu\in\mathbb S^{n-1}$, and a sequence $r_j\to0^+$ such that
$$
u_{r_j,z_0}(x,t)\longrightarrow A(x\cdot\nu)_+^\alpha,
$$
locally uniformly in $\R^n\times\R$.
The respective set of regular points is denoted by $\mathcal R_u$.
\end{definition}

Neither definition assumes uniqueness of blow-ups. However, by \cite[Proposition 4.5 and Theorem 6.4]{AFSW22}, the nutrient half-space blow-up at points $z_0 \in\mathcal R_v$ is unique. Furthermore, \cite[Proposition 4.10]{AFSW22} states that $\mathcal R_v$ is relatively open. Below, we show that at points of $\Gamma\cap\mathcal R_v$, the regular nutrient blow-up determines the forcing term in the limiting equation for the density defect. In particular, every blow-up of $u$ is a nonnegative entire solution of a linear parabolic equation with an explicit half-space source and the prescribed intrinsic global growth.

\begin{lemma}\label{blow-up limit}
If $(u,v)\ge0$ is a bounded solution of \eqref{NewSystem}, and $z_0\in\Gamma\cap\mathcal R_v$, then for every sequence $r_j\downarrow0$, a subsequence of
$u_{r_j,z_0}$ converges locally uniformly on
$\R^n\times\R$ to some $U\ge0$ satisfying
\begin{equation}\label{the RHS}
\Delta U-\partial_tU=c_\beta^\gamma(x\cdot\nu)_+^{\alpha-2}
\quad\text{in }\R^n\times\R,
\end{equation}
where $\nu=\nu(z_0)$ is the nutrient blow-up direction.
Moreover, for every $R>0$,
\begin{equation}\label{globalgrowth}
\sup_{Q_R}U\le CR^\alpha.
\end{equation}
\end{lemma}

\begin{proof}
By Proposition~\ref{two sided growth}, for every fixed $R>0$,
$$
\sup_{Q_R}u_{r,z_0}\le CR^\alpha, \qquad\sup_{Q_R}v_{r,z_0}\le CR^\beta,
$$
for all sufficiently small $r>0$. The equations and local parabolic regularity give precompactness on each fixed two-sided cylinder. By a diagonal argument, $u_{r_j,z_0}\to U\ge0$, locally uniformly on $\R^n\times\R$ along a subsequence. Since $z_0\in\Gamma\cap\mathcal R_v$, the uniqueness of the nutrient half-space blow-up at regular points yields
$$
v_{r_j,z_0}\longrightarrow
c_\beta(x\cdot\nu)_+^\beta,
\quad\text{locally uniformly on }\R^n\times\R.
$$
Passing to the limit in
$$
\Delta u_{r_j,z_0}-\partial_tu_{r_j,z_0}=v_{r_j,z_0}^\gamma,
$$
gives \eqref{the RHS}, since $\alpha-2=\beta\gamma$. The bound \eqref{globalgrowth} follows by passing to the limit
in the rescaled two-sided optimal growth estimate.
\end{proof}

\begin{remark}\label{rafa function}
The explicit nutrient profile determines a natural particular solution of the first equation. In fact, since $\alpha>2$, the function
$P(x):=c_{\alpha,\beta}(x\cdot\nu)_+^\alpha$ is $C^2$ across $\{x\cdot\nu=0\}$. Thus,
\begin{equation}\label{P equation}
    \Delta P-\partial_t P=c_{\alpha,\beta}\alpha(\alpha-1)(x\cdot\nu)_+^{\alpha-2}
\end{equation}
in the whole space. Consequently, for $U$ as in Lemma \ref{blow-up limit}, the function
$$
    H:=U-P
$$
is a caloric function in $\R^n\times\R$ and satisfies
$$
    \sup_{Q_R}|H|\le CR^\alpha, \quad \mbox{for each } \; R>0.
$$
\end{remark}
\begin{lemma}\label{caloric liouville}
Let $H$ be an ancient caloric function on $\R^n\times(-\infty,0]$. Assume that, for some $C>0$ and $d\ge0$,
$$
\sup_{Q_R^-}|H|\le C(1+R^d),
$$
for each $R>0$. Then $H$ is a polynomial of parabolic degree at most $\lfloor d\rfloor$.
\end{lemma}

\begin{proof}
We recall that the \textit{parabolic degree} of a monomial $x^\mu t^k$ is $|\mu|+2k$, and decompose $H$ into its parabolically homogeneous components,
$$
H=\sum_{j=0}^{\lfloor\alpha\rfloor}H_j,
$$
where $H_j(rx,r^2t)=r^jH_j(x,t)$. Assuming $H\not\equiv0$, set
$$
m:=\min\{j:H_j\not\equiv0\}.
$$
Then $r^{-m}H(rx,r^2t)\longrightarrow H_m(x,t)$, locally uniformly as $r\to0^+$. 

We then fix $z_0:=(x_0,t_0)\in\R^n\times(-\infty,0]$, and consider $\mu$ a multi-index and $k\ge0$. Since $H$ is caloric, we have the interior estimate
$$
    |D_x^\mu\partial_t^kH(z_0)|\le CR^{-|\mu|-2k}\sup_{Q_R^-(z_0)}|H|,
$$
for a constant $C>0$, depending only on $\mu$ and $k$. The growth assumption yields
$$
    |D_x^\mu\partial_t^kH(z_0)|\le CR^{d-(|\mu|+2k)},
$$
for each $R>1$, where $C>0$ depends only on $\mu$, $k$ and $z_0$. If $|\mu|+2k>d$, letting $R\to\infty$ gives $D_x^\mu\partial_t^kH(z_0)=0$. Since $z_0$ was arbitrary, all derivatives of parabolic order strictly larger than $d$ vanish. Hence, $H$ is a polynomial of parabolic degree at most $\lfloor d\rfloor$.
\end{proof}

\subsection{A parabolic Liouville theorem}\label{s7}

Next, we establish a Liouville theorem that will later be used to characterize the blow-up profiles at regular points of the nutrient free boundary.

\begin{theorem}\label{rigidity}
Let $U$ be a nonnegative solution of
\begin{equation}\label{ancient solution}
    \Delta U -\partial_t U=c_\beta^\gamma(x\cdot\nu)_+^{\alpha-2}\quad\text{ in }\ \R^n\times(-\infty,0],
\end{equation}
for some $\nu\in\mathbb{S}^{n-1}$, such that 
\begin{equation}\label{ancient solution growth}
    \sup_{Q_R^-}U\le CR^\alpha,
\end{equation}
for each $R>0$. Then
$$
U(x,t)=c_{\alpha,\beta}(x\cdot\nu)_+^\alpha.
$$
\end{theorem}

\begin{proof}
By Remark \ref{rafa function}, we consider $P(x)=c_{\alpha,\beta}(x\cdot\nu)_+^\alpha$ and $H=U-P$,  where $H$ is ancient and caloric. Furthermore, for each $R>0$,
$$
    \sup_{Q_R^-}|H|\le CR^\alpha.
$$
In particular,
\begin{equation}\label{H growth}
    \sup_{Q_r^-}|H|\le Cr^\alpha,\quad\text{ as } r\to0.
\end{equation}
By Lemma~\ref{caloric liouville}, $H$ is a caloric polynomial with parabolic degree at most $\lfloor \alpha\rfloor$. 

\medskip

\textsc{Case 1.} Suppose first that $\lfloor\alpha\rfloor<\alpha$. Assuming $H\not\equiv0$, consider $m$ the smallest parabolic degree of a nonzero term $H_m$. Then, $m\le\lfloor\alpha\rfloor<\alpha$, and
$$
    r^{-m}H(rx,r^2t)\to H_m(x,t),\quad\text{as}\ r\to0.
$$
The latter, recalling \eqref{H growth}, gives
$$
    \sup_{Q_1^-}|r^{-m}H(rx,r^2t)|\le Cr^{\alpha-m}\to0,
$$
which contradicts $H_m\not\equiv0$.

\medskip

\textsc{Case 2.} Assuming $\alpha\in\mathbb N$, since $\alpha>2$, we have $\alpha\ge3$. From \eqref{H growth}, every parabolically homogeneous term of $H$ of degree strictly smaller than $\alpha$ is identically zero. Therefore, either $H\equiv0$, or $H$ is a parabolic polynomial of degree $\alpha$, i.e.,
\begin{equation}\label{H homogeneous}
    H(rx,r^2t)=r^\alpha H(x,t).
\end{equation}
We show that the latter is not possible. After a rotation, we may assume that $\nu=e_1$ in the definition of $P$. Since $U\ge0$, and
$$
    P=0\quad\text{ on }\ \{x_1\le0\},
$$
whenever $x_1\le0$ and $t\le0$, we obtain $H(x,t)\ge0$. Setting
$$
h(x):=H(x,-1),
$$
the parabolic homogeneity of $H$ yields
$$
H(x,t)=(-t)^{\frac{\alpha}{2}}
h\left(\frac{x}{\sqrt{-t}}\right),
\qquad t<0.
$$
Since $\Delta H-H_t=0$, direct differentiation
gives
\begin{equation}\label{h equation}
\Delta h-\frac12x\cdot Dh+\frac{\alpha}{2}h=0
\quad\text{in }\R^n.
\end{equation}
Moreover,
\begin{equation}\label{h positive}
        h(x)\ge0\quad\text{in }\; x_1\le0.
\end{equation}
Write $x=(s,x')\in\R\times\R^{n-1}$, set
$$
d\mu_{n-1}(x'):=(4\pi)^{-(n-1)/2}e^{-|x'|^2/4}\,dx',
$$
and define
$$
g(s):=\int_{\R^{n-1}}h(s,x')\,d\mu_{n-1}(x').
$$
Since $h$ is a polynomial, so is $g$. Moreover, \eqref{h positive} implies that
\begin{equation}\label{g positive1}
    g(s)\geq0,\quad \mbox{for }\; s\leq0.
\end{equation}
Observe that
$$
\int_{\R^{n-1}}
\left(\Delta_{x'}h-\frac12x'\cdot D_{x'}h\right)
\,d\mu_{n-1}=0,
$$
therefore, integrating \eqref{h equation} and using integration by parts, gives
\begin{equation}\label{g equation}
    g''(s)-\frac{s}{2}g'(s)
    +\frac{\alpha}{2}g(s)=0.
\end{equation}
We claim that $g\equiv0$. Indeed, suppose that
$g\not\equiv0$. Comparing the highest-order
terms in \eqref{g equation}, we conclude that
$g$ is a polynomial of degree $\alpha$. Moreover, since $H$ is a parabolically homogeneous polynomial of degree
$\alpha\in\mathbb N$, we can conclude that $h(-x)=(-1)^\alpha h(x)$. Thus, 
$$
g(-s)=(-1)^\alpha g(s).
$$
Suppose first that $\alpha$ is even. Then $g$ is even, and \eqref{g positive1} yields
$$
g\geq0\qquad\text{in }\R.
$$
Set
$$
\mathscr L_1:=\frac{d^2}{ds^2}-\frac{s}{2}\frac{d}{ds},
$$
and \eqref{g equation} can be rewritten as
$$
\mathscr L_1g=-\frac{\alpha}{2}g.
$$
Since $\mathscr L_1$ is self-adjoint in $L^2(d\mu_1)$, 
and $\mathscr L_1 1=0$, we obtain
$$
-\frac{\alpha}{2}\int_{\R}g\,d\mu_1=\int_{\R}g\,\mathscr L_1 1\,d\mu_1=0.
$$
Since $g\ge0$ and $d\mu_1$ has a strictly positive density, we conclude that $g\equiv0$.

Finally, if $\alpha$ is odd, then \eqref{g positive1} gives
$$
    g(s)\le0, \quad\text{for}\ s\ge0,
$$
therefore, 
\begin{equation}\label{nonpositive}
     sg(s)\le0, \quad\text{for}\  s\in\R.
\end{equation}
Since $\mathscr L_1s=-s/2$, and $\mathscr L_1$ is self-adjoint, then
$$
-\frac\alpha2\int_{\R}sg(s)\,d\mu_1=\int_\R s\mathscr L_1g\,d\mu_1=\int_\R g\mathscr L_1s\,d\mu_1=-\frac12\int_\R sg(s)\,d\mu_1.
$$
Since $\alpha\ge3$, this identity gives
$$
    \int_\R sg(s)\,d\mu_1=0;
$$
together with \eqref{nonpositive}, it implies $g\equiv0$. 

Thus, $g\equiv0$ in both cases. On the other hand, for each $s<0$, one has 
$$
h(s,x')\ge0 \quad \mbox{and} \quad
0=g(s)=\int_{\R^{n-1}}h(s,x')\,d\mu_{n-1}(x'),
$$
hence, $h(s,x')=0$, for each  $x'\in\R^{n-1}$ and $s<0$. This means that the polynomial $h$ vanishes in the open half-space $\{x_1<0\}$, and therefore $h\equiv0$, implying $H\equiv0$, and the result follows.
\end{proof}

\subsection*{Blow-up classification}
\label{s8}

We conclude this section by establishing the equivalence of regular points associated with the two components along the joint free boundary. More precisely, we show that the existence of a half-space blow-up for either component
uniquely determines the joint blow-up on $\R^n\times\R$, whose components exhibit the intrinsic homogeneities dictated by $\gamma$ and $\theta$ and share a common positivity half-space.

\begin{theorem}\label{joint blowup}
If $(u,v)\ge0$ is a bounded solution of \eqref{NewSystem}, and $z_0=(x_0,t_0)\in\Gamma\cap\mathcal R_v$, then there exists a unique $\nu=\nu(z_0)\in\mathbb S^{n-1}$ such that
$$
(u_{r,z_0},v_{r,z_0})\longrightarrow
\bigl(c_{\alpha,\beta}(x\cdot\nu)_+^\alpha,
      c_\beta(x\cdot\nu)_+^\beta\bigr),
$$
locally uniformly in $\R^n\times\R$ as $r\downarrow0$.
In particular, the two limiting profiles have the same
positivity half-space $\{x\cdot\nu>0\}$ and the common
interface $\{x\cdot\nu=0\}$.
\end{theorem}

\begin{proof}
Given any sequence $r_j\downarrow0$, by Lemma~\ref{blow-up limit}, we can extract a subsequence such that
$(u_{r_j,z_0},v_{r_j,z_0})\to(U,V)$, locally uniformly on
$\R^n\times\R$, with $V=c_\beta(x\cdot\nu)_+^\beta$, for a uniquely determined $\nu$. By Theorem \ref{rigidity}, applied to $U$ on $t\le0$,
$$
U(x,t)=c_{\alpha,\beta}(x\cdot\nu)_+^\alpha
\qquad(t\le0).
$$
In the whole space-time, the difference
$H:=U-c_{\alpha,\beta}(x\cdot\nu)_+^\alpha$
is caloric and has polynomial growth by \eqref{globalgrowth}. It vanishes for $t\le0$; uniqueness of the forward heat equation in the polynomial-growth class yields $H\equiv0$, also for $t>0$. Every subsequential limit is therefore the same pair,
which proves convergence of the full rescaled family.
\end{proof}

Conversely, at points of $\Gamma\cap\mathcal R_u$, the half-space blow-up of the density defect determines the nutrient profile through the first equation of the system. The nutrient equation then determines its coefficient, yielding the universal half-space profile for $v$ without an additional Liouville argument.

\begin{proposition} \label{equivalence regular points}
If $(u,v)\ge0$ is a bounded solution of \eqref{NewSystem}, then 
$$
\mathcal R_u\subset\Gamma\quad\mbox{and}\quad\mathcal R_u=\Gamma\cap\mathcal R_v.
$$
In particular, every regular point of the density defect free boundary is a joint free boundary point. Moreover, at every such point, the full joint blow-up is uniquely determined and is given by Theorem~\ref{joint blowup}.
\end{proposition}

\begin{proof}
We first show that $\mathcal R_u\subset\Gamma$.
Let $z_0=(x_0,t_0)\in\mathcal R_u$.
By Definition~\ref{regular points u}, there
exist $A>0$, $\nu\in\mathbb S^{n-1}$, and
$r_j\downarrow0$ such that
$$
u_{r_j,z_0}(x,t)
\longrightarrow A(x\cdot\nu)_+^\alpha,
$$
locally uniformly in $\mathbb R^n\times\mathbb R$. Since $\alpha>2$, we have
$$
\frac{u(x_0+r_jx,t_0)}{r_j^2}
=
r_j^{\alpha-2}u_{r_j,z_0}(x,0)
\longrightarrow0,
$$
locally uniformly in $\mathbb R^n$. On the other hand, since $z_0\in\Gamma_u$,
Corollary~\ref{c4.6} gives
$$
\frac{u(x_0+r_jx,t_0)}{r_j^2}
\longrightarrow
\frac12\langle D_x^2u(z_0)x,x\rangle.
$$
Consequently, $D_x^2u(z_0)=0$. Recalling the characterization
$$
\Gamma=\Gamma_u\cap\{D_x^2u=0\}
$$
from Corollary~\ref{c4.7}, we conclude that $z_0\in\Gamma$.

Proposition \ref{coupling prop} gives $\Gamma\subset\Gamma_u\cap\Gamma_v$. Since $c_{\alpha,\beta}>0$, Theorem \ref{joint blowup} gives $\Gamma\cap\mathcal R_v\subset\Gamma\cap\mathcal R_u$. Suppose now $z_0\in\Gamma\cap\mathcal R_u$, and choose the sequence in Definition \ref{regular points u}, so that
$$
u_j:=u_{r_j,z_0}\longrightarrow U:=A(x\cdot\nu)_+^\alpha,
$$
locally uniformly. Set $v_j:=v_{r_j,z_0}$. Then 
$$
\Delta u_j-\partial_tu_j=v_j^\gamma,\qquad\Delta v_j-\partial_tv_j=v_j^\theta.
$$
For every fixed $R>0$, Proposition~\ref{two sided growth} gives
$$
\sup_{Q_R}v_j=r_j^{-\beta}\sup_{Q_{r_jR}(z_0)}v\le CR^\beta,
$$
for all sufficiently large $j$. The same estimate on larger
cylinders bounds $v_j^\theta$. Interior parabolic estimates and a diagonal argument therefore yield a subsequence with $v_j\to V\ge0$, locally uniformly on $\R^n\times\R$. Thus,
$$
\Delta U-\partial_tU=V^\gamma, \qquad\Delta V-\partial_tV=V^\theta.
$$
The first equation yields
$$
V^\gamma=A\alpha(\alpha-1)(x\cdot\nu)_+^{\alpha-2}.
$$
Using $\gamma>0$, $V\ge0$, and $\alpha-2=\gamma\beta$, we obtain
$$
V=C_0(x\cdot\nu)_+^\beta,
\qquad C_0=[A\alpha(\alpha-1)]^{1/\gamma}>0.
$$
These identities also hold at $t=0$ by continuity. On the open half-space $\{x\cdot\nu>0\}$, substitution into the second equation and the fact that $\beta-2=\theta\beta$ yield
$$
C_0\beta(\beta-1)=C_0^\theta.
$$
Hence, we get $C_0=c_\beta$ and, similarly, $A=c_\beta^\gamma/(\alpha(\alpha-1))=c_{\alpha,\beta}$.

The extracted subsequence is therefore a nutrient half-space blow-up, so $z_0\in\mathcal R_v$. This proves the reverse inclusion. Uniqueness of the direction and convergence of the full joint family now follow from Theorem \ref{joint blowup}.
\end{proof}

In view of Proposition \ref{equivalence regular points}, we define the \emph{common regular set} by
\begin{equation}\label{common regular set}
\mathcal R:= \mathcal R_u=\Gamma\cap\mathcal R_v.
\end{equation}

\begin{remark}\label{backward two sided equivalence}
At a joint free boundary point, requiring convergence in $\R^n\times\R$ in Definitions \ref{regular points} and \ref{regular points u} does not exclude points that would be regular under the earlier \emph{backward} subsequential definition. Indeed, suppose $v_{r_j,z_0}$ tends to a nutrient half-space profile on $t\le0$. Proposition \ref{two sided growth} gives two-sided subsequential compactness. Any two-sided limit agrees with that stationary profile for $t\le0$ and hence for $t\ge0$ by the uniqueness of the scalar forward equation 
$$V_t-\Delta V=-V^\theta$$ 
in the polynomial-growth class. Thus, the same sequence has a two-sided half-space limit. Similarly, suppose that $u_{r_j,z_0}$ tends to $A(x\cdot\nu)_+^\alpha$, $A>0$, on $t\le0$. After extracting a two-sided limit of $v_{r_j,z_0}$, the first equation forces the nutrient limit on $t\le0$ to equal $C(x\cdot\nu)_+^\beta$, where $C=[A\alpha(\alpha-1)]^{1/\gamma}$. The nutrient equation forces $C=c_\beta$, and forward uniqueness for the nutrient equation yields the same profile on $t\ge0$. The first equation then uniquely determines the forward density-defect limit, because its difference from the stationary half-space profile is caloric and vanishes at $t=0$. Consequently, the two-sided and backward subsequential notions agree along $\Gamma$.
\end{remark}

\section{Regularity of the joint free boundary} \label{s10}

In this section, we study the regularity of the joint free
boundary near points of the common regular set $\mathcal R$. Our main result establishes that, at these points, the joint free boundary locally coincides with the nutrient free boundary in a backward parabolic neighborhood. Hence, the joint interface inherits the spatial $C^{1,\sigma}$ and temporal $C^{0,1/2}$ regularity of the nutrient interface obtained in \cite{AFSW22}.

\subsection*{The backward zero cone}\label{s9} We derive a quantitative vanishing criterion for the nutrient equation, which allows us to construct backward non-tangential cones where $v$ vanishes identically.
Within these regions, the density defect $u$ is caloric. Exploiting its superquadratic growth estimate through a suitable parabolic barrier, we show that $u$ must also vanish in the nutrient-depleted phase near the joint boundary.

\begin{lemma}\label{lemma1}
If $(u,v)\ge0$ is a bounded solution of \eqref{NewSystem}, and for some $\overline{Q_r^-(z)}\subset Q_1$, one has
$$
\sup_{Q_r^-(z)}v\leq \frac{c_*}{2}r^\beta,
$$
where $c_*>0$ and $\beta>0$ are defined in
\eqref{lower estimate for barrier} and
\eqref{definition of beta}, respectively, then
$$
v\equiv0\quad\text{in} \ Q_{r/8}^-(z).
$$
\end{lemma}

\begin{proof}
By translation and scaling, it is enough to consider $r=1$ and $z=(0,0)$. Let $\psi$ be the barrier function defined in
\eqref{the barrier function}. As established in
Section \ref{s5}, it satisfies
$$
\begin{array}{rlr}
\Delta\psi-\psi_t\leq\psi^\theta
& \text{in} & Q_1^-, \\
\psi\geq c_*
 & \text{on} & \partial_pQ_1^-.
\end{array}
$$
Consequently, $v \leq\psi$ on $\partial_pQ_1^-$. By the parabolic comparison principle, $v\leq\psi$ in $Q_1^-$. On the other hand, by the construction of the barrier, $\psi\equiv0$ in $Q_{1/8}^-$. Since $v\geq0$, we derive
$$
v\equiv0\quad\text{in} \ Q_{1/8}^-,
$$
as desired.
\end{proof}

Next, we establish the existence of a backward zero cone
at regular nutrient points.

\begin{lemma}\label{lemma2}
If $(u,v)\ge0$ is bounded solution of \eqref{NewSystem}, and $z_0=(x_0,t_0)\in\Gamma\cap\mathcal R_v$, then for each $\varepsilon\in(0,1)$, there exists
$\rho_\varepsilon>0$ such that $v(x,t)=0$ whenever the following conditions
$$
t\leq t_0,\qquad
0<d(x,t)<\rho_\varepsilon, \qquad \mbox{and} \quad 
-(x-x_0)\cdot\nu(z_0)\geq\varepsilon d(x,t)
$$
hold, where $d(x,t):=d_p((x,t),z_0)$.
\end{lemma}

\begin{proof}
By translation and rotation, we may assume that
$z_0=(0,0)$ and $\nu(z_0)=e_1$. Since $z_0\in\mathcal R_v$, the uniqueness of the nutrient blow-up yields
\begin{equation}\label{0 is a regular point}
v_r(x,t):=r^{-\beta}v(rx,r^2t)
\longrightarrow c_\beta(x_1)_+^\beta,
\end{equation}
locally uniformly in $\R^n\times(-\infty,0]$ as $r\to0^+$.

\medskip

\textsc{Step 1.}  We claim that for every compact set
$$
K\subset\{(x,t):x_1<0,\ t\leq0\},
$$
there exists $r_K>0$ such that
$$
v_r\equiv0\quad\text{on} \ K,
$$
for every $0<r<r_K$. 

Indeed, choose $\delta>0$ so small that
$$
K':=\bigcup_{z\in K}\overline{Q_\delta^-(z)}
$$
is a compact subset of $\{x_1<0,\ t\leq0\}$. Since the limiting profile in \eqref{0 is a regular point} vanishes identically on $K'$, local uniform convergence gives
$$
\sup_{K'}v_r\longrightarrow0
\qquad\text{as }r\to0^+.
$$
Consequently, for all sufficiently small $r>0$,
$$
\sup_{Q_\delta^-(z)}v_r
\leq\frac{c_*}{2}\delta^\beta,
\qquad z\in K.
$$
Applying Lemma \ref{lemma1} to $v_r$ in each cylinder
$Q_\delta^-(z)$, we obtain
$$
v_r\equiv0\quad\text{in} \ Q_{\delta/8}^-(z),\quad z\in K.
$$
In particular, $v_r(z)=0$, for every $z\in K$,
which proves the claim.

\medskip

\textsc{Step 2.}  Next, we construct the backward zero cone. Fix $\varepsilon\in(0,1)$ and consider
$$
K_\varepsilon:=\left\{(x,t)\in\R^n\times(-\infty,0]:
d_p((x,t),(0,0))=\frac12,\ x_1\leq-\frac{\varepsilon}{2}
\right\}.
$$
Clearly, $K_\varepsilon$ is a compact subset of $\{x_1<0,\ t\leq0\}$. By Step 1, there exists $r_\varepsilon>0$ such that
$$
v_r\equiv0\quad\text{on} \ K_\varepsilon
$$
whenever $0<r<r_\varepsilon$. Set $\rho_\varepsilon:=r_\varepsilon/2$. Let $(x,t)$ satisfy
$$
t\leq0,\quad0<d(x,t)<\rho_\varepsilon,\quad-x_1\ge\varepsilon d(x,t),
$$
where $d(x,t)=d_p((x,t),(0,0))$. Taking $r:=2d(x,t)$, we have $0<r<r_\varepsilon$,
$$
d_p\left(\left(\frac{x}{r},\frac{t}{r^2}\right),(0,0)
\right)=\frac12, \quad \mbox{and} \quad -\frac{x_1}{r}
\geq\frac{\varepsilon d(x,t)}{2d(x,t)}=\frac{\varepsilon}{2}.
$$
Therefore,
$$
\left(\frac{x}{r},\frac{t}{r^2}\right)
\in K_\varepsilon.
$$
By the definition of $v_r$ and the vanishing property established above,
$$
v(x,t)=r^\beta v_r\left(\frac{x}{r},\frac{t}{r^2}\right)=0.
$$
This proves the result.
\end{proof}

\subsection*{Regularity of the joint free boundary}

We conclude this section by establishing the regularity of the joint free boundary near points in $\mathcal R$. More precisely, we show that the joint and nutrient free boundaries coincide locally in a backward parabolic neighborhood, and consequently, the joint interface is $C^{1,\sigma}$ in space and half-H\"older in time. 

The following result adapts \cite[Theorem 6.5]{AFSW22} to our setting.

\begin{theorem}\label{AFSW regularity}
If $(u,v)\ge0$ is a bounded solution of \eqref{NewSystem} and $z_0=(x_0,t_0)\in\Gamma\cap\mathcal R_v$, then there exist $\rho>0$ and $\sigma\in(0,1)$ such that its ordinary space-time nutrient free boundary $\partial\{v>0\}\cap Q_\rho(z_0)$ is, up to a spatial rotation,
a graph
$$
x_n=f(x',t),
$$
where $f$ is $C^{1,\sigma}$ in the spatial variables and
$C^{0,1/2}$ in time. 

Furthermore, on a smaller two-sided rectangular neighborhood $\mathcal B$ of $z_0$, orienting the normal toward the positive phase,
\begin{equation}\label{two sided nutrient phases}
\{v=0\}\cap\mathcal B=\{x_n\le f(x',t)\}\cap\mathcal B,
\quad\{v>0\}\cap\mathcal B=\{x_n>f(x',t)\}\cap\mathcal B.
\end{equation}
In particular, the trace of this graph in a backward cylinder represents the backward nutrient boundary $\Gamma_v$.
\end{theorem}

\begin{proof}
Since $v\ge0$ and $\theta\in(0,1)$, the vector $\mathbf w=(v,0)$ solves the system studied in \cite{AFSW22}, for $q=\theta$ and $\kappa=\beta$. By the optimal growth and nondegeneracy estimates, $v$ has vanishing order precisely $\beta$ at every point of $\Gamma$. Moreover, assuming $z_0\in \mathcal{R}_v$, according to Definition \ref{regular points}, provides the remaining hypothesis in the definition of a regular point in \cite[Definition 4.4]{AFSW22}. The claimed regularity of the nutrient free boundary then follows directly from \cite[Theorem 6.5]{AFSW22}.
\end{proof}

Finally, we are ready to prove the regularity of the joint free boundary near points in $\mathcal R$. Although the nutrient interface $\Gamma_v$ is regular in a neighborhood of such points, this regularity does not automatically extend to the joint free boundary $\Gamma$. We overcome this difficulty by proving that the two interfaces coincide in a backward parabolic neighborhood, thereby establishing the claimed regularity of the joint free boundary on that backward neighborhood.
\begin{theorem}\label{joint FB regularity}
If $(u,v)\ge0$ is a bounded solution of \eqref{NewSystem} and $z_0\in\mathcal R$, then there exists $r_0>0$ such that
\begin{equation}\label{important coupling}
    \{v=0\}\cap Q_{r_0}^-(z_0)\subset\{u=0\}.
\end{equation}
Consequently, 
$$
\Gamma\cap Q_{r_0}^-(z_0)=\Gamma_v\cap Q_{r_0}^-(z_0),
$$
and the joint free boundary is $C^{1,\sigma}$ in the spatial variables, for some $\sigma\in(0,1)$, and $C^{0,1/2}$ in time.
\end{theorem}
\begin{proof}
By \eqref{common regular set}, $z_0\in\mathcal R_v$. Recalling Definition \ref{regular points} and Theorem \ref{joint blowup}, we assume, without loss of generality, $z_0=(0,0)$ and $\nu= e_1$. 

\medskip

\textsc{Step 1.} Since $z_0\in\Gamma\cap\mathcal R_v$,
Theorem \ref{AFSW regularity} gives a continuous graph
$x_1=f(x',t)$ representing the nutrient boundary in a sufficiently
small backward neighborhood of the origin, with $f(0,0)=0$.
The graph is spatially $C^{1,\sigma}$, for some $\sigma\in(0,1)$,
and $C^{0,1/2}$ in time. We restrict the two-sided graph supplied by Theorem~\ref{AFSW regularity}
to $t\le0$. We choose $a$, $b$, $T>0$ small enough such that, in 
$$
B:=(-a,a)\times B_b'\times(-T,0],
$$
the interface is precisely that graph and $|f|<a/4$. 

\medskip

\textsc{Step 2.} The nutrient phases have already been
identified on both sides of the graph in
\eqref{two sided nutrient phases}. Restricting that identity to the backward cylinder $B$ in Step 1 gives
    \begin{equation}\label{10.2}
    \begin{aligned}
      \{v=0\}\cap B&=\{x_1\le f(x',t)\}\cap B,\\
      \{v>0\}\cap B&=\{x_1>f(x',t)\}\cap B.
    \end{aligned}
    \end{equation}

\medskip
    
\textsc{Step 3.} Set 
    $$
    D:=\{(x,t)\in B;\,x_1<f(x',t)\}.
    $$
    By \eqref{10.2}, $u\ge0$ is a caloric function in $D$. Define
    $$
    B_*:=(-a/2,a/2)\times B'_{b/2}\times(-T/2,0],
    $$
    and assume that there exists $(y,s)\in D\cap B_*$ such that $u(y,s)>0$, with $s<0$. We claim that, for every $(x,t)\in D\cap B_*\cap\{s<t\le0\}$, there is a path from $(y,s)$ to $(x,t)$ in $D$ whose time coordinate strictly increases. Indeed, for a short interval immediately after $s$, move from $y$ to $(-3a/4, y')$ by decreasing the first coordinate. During the middle interval, stay at the first coordinate $-3a/4$ and move tangentially to $x'$. During a short interval ending at $t$, move from $(-3a/4,x')$ to $x$. The strict graph inequalities at the endpoints and the continuity of $f$ preserve the first and last movements in $D$ when their time intervals are short. The middle part lies in $D$ at all times because $-3a/4<f$. Positivity propagates along this path. We then cover the compact path with finitely many spatial balls with short time intervals, arranged in time. In the first ball, choose a nonzero nonnegative smooth function supported in $\{u(\cdot,s)>0\}$ and bounded above by $u(\cdot,s)$. By the strong maximum principle, its zero-boundary heat evolution is positive in the ball at every later time and is bounded above by $u$, yielding $u(x,t)>0$. This argument also reaches a target at time zero using an interior ball over a backward time interval ending at zero.

\medskip

    \textsc{Step 4.} Next, we construct a quadratic subsolution with an exact zero cone. Namely, set
    $$
    \omega(x,t):=x_1^2-\frac{|x'|^2}{4n}+t.
    $$
    Direct calculation gives
    $$
    \Delta\omega-\omega_t=\frac{n+1}{2n}>0.
    $$
    For $x_1<0$, $t\le0$ and $\omega\ge0$, we have
    $$
    x_1^2\ge\frac{|x'|^2}{4n}-t,\quad-x_1\ge\frac{|x|}{(1+4n)^{1/2}}\quad\text{and}\quad-x_1\ge\sqrt{-t}.
    $$
    Hence, 
    \begin{equation}\label{10.3}
        -x_1\ge\frac{1}{(1+4n)^{1/2}}d(x,t).
    \end{equation}
    Using Lemma \ref{lemma2} with $\varepsilon=1/(2(1+4n)^{1/2})$, we conclude that, in a neighborhood of points satisfying \eqref{10.3}, one has $v\equiv0$, and hence, such points are in $D$ and not on its graph boundary, thanks to \eqref{10.2}.

    Next, we choose $R>0$ so small that $\overline{Q_R^-}\subset B_*$ and $R^2<-s$. Set 
    $$
    \Omega_R:=\{(x,t):|x|<R,\,-R^2<t\le0,\,x_1<0,\,\omega(x,t)>0\}\subset D.
    $$
    Then $u$ is caloric in $\Omega_R$, and
    $$
    \Gamma_R:=\{(x,t):|x|=R,\,-R^2\le t\le0,\,x_1\le0,\,\omega(x,t)\ge0\}
    $$
    is a compact subset of $D\cap B_*$. It cannot contain a point with $x_1=0$, since then $\omega\ge0$ would imply $x'=0$ and $t=0$, contradicting $|x|=R$. For $0<R<\rho_\varepsilon$, where $\rho_\varepsilon$ is as in Lemma \ref{lemma2}, the latter and \eqref{10.3} imply that the set $\Gamma_R\setminus\{(0,0)\}$ lies in the nutrient depleted interior. Observe also that since $R^2<-s$, every point in $\Gamma_R$ is later than $s$. Recalling Step 3 above, we conclude
    $$
    m_R:=\min_{\Gamma_R}u>0.
    $$
    Let 
    $$
    \eta:=\frac{m_R}{R^2}.
    $$
    Since $\omega\le R^2$ in $\Gamma_R$, then $\eta\omega\le u$. On $\omega=0$, the same inequality follows from $u\ge0$. At $t=-R^2$ there is no positive part of $\omega$ in $B_R$, as
    $$
    \omega(x,-R^2)\le|x|^2-R^2\le0.
    $$    
    In $\Omega_R\cap\{t\le\tau\}$, $\tau<0$, the function $F:=\eta\omega-u$ satisfies
    $$
    \Delta F-F_t=\frac{n+1}{2n}\eta>0,
    $$
    and is nonpositive on the parabolic boundary. If it had a positive maximum, that maximum would occur at a spatial interior point, yielding $\Delta F-F_t\le0$, a contradiction. Letting $\tau\to0-$, we conclude that   
    $$    
    u\ge\eta\omega\quad\text{in}\ \Omega_R.
    $$

\medskip

    \textsc{Step 5.} For $0<h<R$, the last inequality in Step 4 gives
    $$
    u(-he_1,0)\ge\eta h^2.
    $$
    On the other hand, by Theorem~\ref{optimal growth theorem} applied with $r=2h$, one has
    $$
    u(-he_1,0)\le C(2h)^\alpha,
    $$
    where $\alpha$ is defined by \eqref{definition of alpha} as before. Thus, 
    $$
    0<\eta\le C2^\alpha h^{\alpha-2},
    $$
    for every sufficiently small $h$, which is impossible as $h\to0$. Consequently, there is no point in $D\cap B_*$ where $u>0$. Hence, $u=0$ in $D\cap B_*$. Since each graph point can be approached by points strictly below the graph in $B_*$, by continuity, $u=0$ also on the graph. Recalling \eqref{10.2}, this proves \eqref{important coupling}. 

\medskip

    \textsc{Step 6.} To complete the proof, note that \eqref{important coupling} implies locally $\{u>0\}\subset\{v>0\}$, and $u=0$ on the nutrient interface. Hence Proposition \ref{coupling prop} gives, locally, $\Gamma=\Gamma_v$. The claimed regularity for the joint free boundary now follows from Theorem \ref{AFSW regularity}.
\end{proof}

\begin{remark}
Even though the system and the nutrient regularity theorem are formulated on two-sided cylinders, Theorem \ref{joint FB regularity} asserts coincidence only in $Q_{r_0}^-(z_0)$. On the other hand, the phase separation of the \emph{nutrient} interface in \eqref{two sided nutrient phases} is two-sided. Extending coincidence of \emph{the joint} interface to $Q_{r_0}(z_0)$ would require a separate argument.
\end{remark}

\begin{corollary}
    If $(u,v)\ge0$ is a bounded solution of \eqref{NewSystem}, and $z_0=(x_0,t_0)\in\mathcal R$, then
    $$
    \lim_{r\to0}\frac{|E\cap Q_r^-(z_0)|}{|Q_r^-|}=\lim_{r\to0}\frac{|E_{t_0}\cap B_r(x_0)|}{|B_r|}=\frac12,
    $$
    where $E$ is $\{u>0\}$, $\{v>0\}$ or $\{|(u,v)|>0\}$, and $E_{t_0}:=\{x;\,(x,t_0)\in E\}$.    
\end{corollary}
\begin{proof}
Without loss of generality, we may assume $z_0=(0,0)$ and $\nu=e_1$. At every point of $Q_1^-$ with $x_1>0$, Theorem \ref{joint blowup} yields $u_r>0$ and $v_r>0$, for $r>0$ sufficiently small. At every point with $x_1<0$, Step 1 of the proof of Lemma \ref{lemma2} gives $v_r=0$, and hence, by \eqref{important coupling}, also $u_r=0$. Thus, the characteristic functions of all three rescaled positivity sets converge almost everywhere on $Q_1^-$ to $\chi_{\{x_1>0\}}$. The Dominated Convergence Theorem then gives that the first limit is indeed $1/2$. The same argument, restricted to time zero and combined with the spatial Dominated Convergence Theorem, yields the second equality.
\end{proof}

\begin{remark}
    When $\theta=0$, the second equation in \eqref{NewSystem} is interpreted as the classical parabolic obstacle problem, for which a rich regularity theory has been developed (see, for example, \cite{B06, CPS04, FRS24, KN78, LM15, MR26}).
\end{remark}

\begin{center}
\begin{figure}[H]
\begin{tikzpicture}[
    scale=1.3,
    line cap=round,
    line join=round
]

\def\cebola{
    (0,1.6)
    .. controls (-0.10,0.85) and (-1.35,0.55)
       .. (-1.35,-0.45)
    .. controls (-1.35,-1.30) and (-0.65,-1.65)
       .. (0,-1.65)
    .. controls (0.65,-1.65) and (1.35,-1.30)
       .. (1.35,-0.45)
    .. controls (1.35,0.55) and (0.10,0.85)
       .. (0,1.6)
    -- cycle
}

\def\linha{
    (0,1.6)
    .. controls (0,2.05) and (0.15,2.40)
       .. (0.10,2.85)
}

\fill[gray!15] \cebola;
\fill[
    pattern=north east lines,
    pattern color=black
] \cebola;

\draw[gray!65, line width=4pt] \cebola;
\draw[gray!65, line width=4pt] \linha;

\draw[black, line width=0.8pt] \cebola;

\node[anchor=west] at (0.75,2.45)
    {$\Gamma_u\setminus \Gamma  \subset \{v>0\}$};
\draw[black, thin, ->]
    (0.8,2.25) -- (0.08,2.0);
    
\node[anchor=west] at (1.85,0.0)
{$\interior\{u=0\}\subset\interior\{v=0\}$ };
\draw[black, thin, ->]
    (1.9,-0.15) -- (0.55,-0.8);

\node[anchor=west] at (0.75,1.45)
    {$\Gamma$};
\draw[black, thin, ->]
    (0.9,1.25) -- (0.55,0.8);

\node at (3.6,-1.4) {$\{u>0\}\cap \{v>0\}$};
%

\draw[gray!35!black] 
(-3.7,3.2) -- (-3.7,-2) -- (5.6,-2) -- (5.6,3.2) -- (-3.7,3.2);

\draw[gray!100, fill=gray!80, opacity=0.3] (-1.0,0.4) circle (0.6cm);

\filldraw [-, white, opacity=0.7, dashed] (-1.0,0.4) -- (-1.6,0.5);

\filldraw[white!80] (-1.0,0.4) circle (0.6pt);

\filldraw[gray!80] (-2.3,1.8) circle (0.0pt) node[below] {{\small $\Gamma\cap Q_r^-(z_0)\subset\mathcal R$}};

\draw[black, thin, ->]
    (-1.4,1.35) -- (-1.03,0.45);

\end{tikzpicture}
\caption{An illustration near a common regular point
$z_0\in\mathcal R=\mathcal R_u=\Gamma\cap\mathcal R_v$. Regularity can be detected through a half-space blow-up of either component. By Theorem \ref{joint FB regularity}, $\Gamma\cap Q_r^-(z_0)=\Gamma_v\cap Q_r^-(z_0)$, and $\{v=0\}\cap Q_r^-(z_0)\subset\{u=0\}$, for $r>0$ small.}
\end{figure}
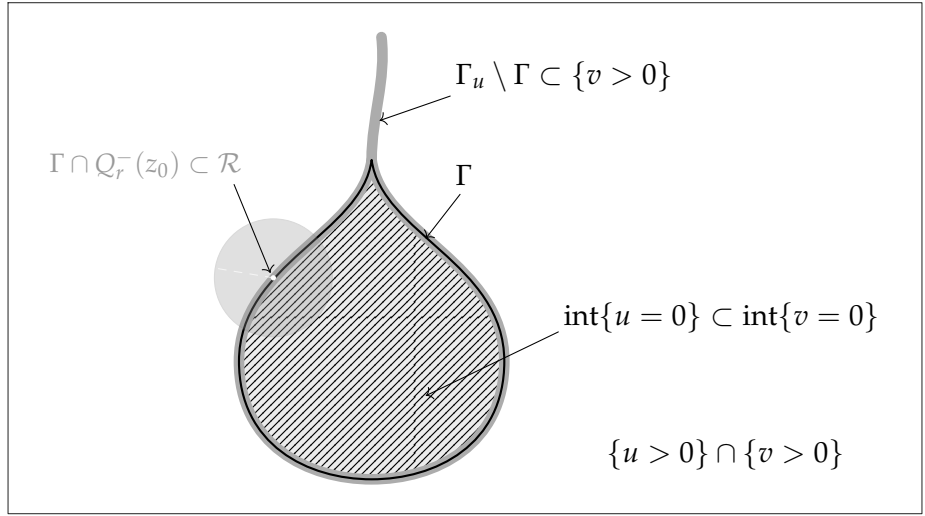
\end{center}

\medskip

{\small \noindent{\bf Acknowledgments.} This publication is based upon work supported by King Abdullah University of Science and Technology (KAUST) under Award No. ORFS-CRG12-2024-6430. DJA is supported by the Conselho Nacional de Desenvolvimento Cient\'\i fico e Tecnol\'ogico (CNPq), under grants 420014/2023-3 and 310542/2026-0.}

\medskip

{\small \noindent{\bf AI disclosure.} The mathematical ideas, statements, and proofs were developed entirely by the authors. AI tools were used to correct typographical errors and improve the language.}

\medskip

{\small \noindent{\bf Data availability statement.} Data sharing is not applicable to this article as no datasets were generated or analyzed during the current study.}

\end{document}